\documentclass[11pt,leqno]{article}

\usepackage{amssymb,amsmath,amsthm,mysects,url,rotating}
 \usepackage{graphicx,epsfig}
 \usepackage{xcolor,graphicx}
\newcommand{\n}{\noindent}

\newcommand{\vp}{\varepsilon}
\newcommand{\bb}[1]{\mathbb{#1}}
\newcommand{\cl}[1]{\mathcal{#1}}

\newcommand{\ovl}{\overline}

\theoremstyle{plain}
\newtheorem{thm}{Theorem}[section]
\newtheorem{lem}[thm]{Lemma}

\newtheorem{pro}[thm]{Proposition}

\newtheorem{cor}[thm]{Corollary}

\theoremstyle{definition}

\newtheorem{dfn}[thm]{Definition}

\theoremstyle{remark}
\newtheorem{rem}[thm]{Remark}

\numberwithin{equation}{section}

\def\RR{\bb R}
\def\CC{\bb C}

\def\E{\bb E}

\def\P{\bb P}
\def\T{\bb T}

\def\d{\delta}
\def\NN{\bb N}
\def\RR{\bb R}

\def\CC{\bb C}

\begin{document}

\title{Quantum Expanders and Geometry of Operator Spaces}

\author{by\\
Gilles Pisier\footnote{Partially supported   by ANR-2011-BS01-008-01.}\\
Texas A\&M University\\
College Station, TX 77843, U. S. A.\\
and\\
Universit\'e Paris VI\\
Inst. Math. Jussieu, Equipe d'Analyse, Case 186, 75252\\
Paris Cedex 05, France}

 \maketitle

\begin{abstract} We show that there are well separated families of quantum expanders with asymptotically the maximal cardinality allowed by a known upper bound.
 This has applications to the ``growth" of certain operator spaces:
It implies asymptotically  sharp estimates for the growth of the multiplicity of $M_N$-spaces needed to represent (up to a constant $C>1$) the $M_N$-version of the $n$-dimensional operator Hilbert space $OH_n$ as a direct sum of copies of $M_N$. We show that, when $C$ is close to 1, this multiplicity grows as $\exp{\beta n N^2}$ for some constant $\beta>0$. The main idea is to relate quantum expanders with "smooth" points on the matricial analogue of the Euclidean unit sphere. This generalizes to operator spaces a classical geometric result on $n$-dimensional Hilbert space (corresponding to N=1).  In an appendix, we give a quick proof of an inequality (related to Hastings's previous work) on random unitary matrices that is crucial for this paper. 
\end{abstract}

\n The term ``Quantum Expander" is used by Hastings in \cite{Ha}
and   by Ben-Aroya and Ta-Shma in  \cite{BA1} to designate
a sequence $\{U^{(N)}\mid N\ge 1 \}$ of  $n$-tuples $ U^{(N)}=(U_1^{(N)},\cdots,U_n^{(N)})$
of $N\times N$ unitary  matrices such that 
there is an $\vp >0$ satisfying the following ``spectral gap" condition:
 \begin{equation}\label{in1} \forall  N\ \forall x\in M_N\quad \| \sum\nolimits_1^n U_j^{(N)} (x-N^{-1}{tr}(x)) { U_j^{(N)}}^* \|_2 \le n(1-\vp)\| x-N^{-1}{tr}(x)\|_2,\end{equation}
where $\|.\|_2$ denotes the Hilbert-Schmidt norm on $M_N$. More generally,
the term is extended to the case when   this is only defined  for infinitely many   $N$'s, and also
to $n$-tuples of matrices satisfying merely $\sum U_j^{(N)}{U_j^{(N)}}^*=\sum{ U_j^{(N)}}^*U_j^{(N)}=nI$.\\
 We will say that an $n$-tuple $U^{(N)}$ satisfying \eqref{in1} is a  $\vp$-quantum expander. We refer the reader to the survey \cite{BA2} for more information and references on quantum expanders.

 In analogy with the classical expanders (see below),  one seeks to exhibit 
 (and hopefully to construct explicitly ) sequences
  $\{U^{(N_m)}\mid m\ge 1 \}$ of $n$-tuples of $N_m\times N_m$ unitary  matrices  
   that are $\vp$-quantum expanders with $N_m\to \infty$ while $n$ and $\vp>0$ remain fixed.
  
 When $G$ is a finite group generated by $S=\{t_1,\cdots,t_n\}$
 the associated Cayley graph $\cl G(G,S)$  is said to have  a spectral gap if
 the left regular representation  $\lambda_G$ satisfies
   \begin{equation}\label{in0bis}\|\sum \lambda_G(t_j)_{|{\bb I}^\perp}\|< n(1-\vp)
   \end{equation}
    where ${\bb I}$ denotes the constant
 function $1$ on $G$.
 Obviously, this is equivalent 
 to the condition that  the unitaries $U_j=\lambda_G(t_j)$ satisfy \eqref{in1}
 when restricted to  diagonal matrices $x$ (here $N=|G|$).
 In this light, quantum expanders appear as a non-commutative version of the classical ones.\\
 More precisely, \eqref{in0bis} holds iff  the unitaries $U_j=\lambda_G(t_j)$ satisfy \eqref{in1}
  for all $x$ in the orthogonal complement of right translation operators. This is easy to deduce from
  the decomposition into irreducibles of $\lambda_G\otimes \bar \lambda_G$, in which the component of the  trivial
  representation corresponds to the restriction to right translation operators.
  
  In addition, for any irreducible representation $\pi$ of $G$, \eqref{in0bis} implies that the unitaries
  $(\pi(t_j))$ satisfy \eqref{in1} because the non trivial irreducible components
  of the representation $\pi \otimes \bar \pi$ are  contained in $\lambda_G$. See Remark \ref{kaz}
  for more on this.
  
 A  sequence of Cayley graphs $\cl G(G^{(m)},S^{(m)})$ constitutes an expander
 in the usual sense  if  \eqref{in0bis} is satisfied with $\vp>0$ and $n$ fixed while $|G^{(m)}|\to \infty$.

Expanders (equivalently expanding graphs) have been extremely useful, especially (in the applied direction)
since Margulis and Lubotzky-Phillips-Sarnak obtained explicit constructions
(as opposed to random ones). We refer to \cite{Lu,HLW}
for more information and references.
\\They have also been used with great success for operator algebras and in operator theory (see e.g. \cite{V,DS,JP} see also \cite{P4,BO}). In \cite{JP}, is crucially used the fact that 
when the dimensions $N,N'$ 
are suitably different, say if $N$ is much larger than $N'$, and $U^{(N)}$ satisfies   \eqref{in1} 
then $U^{(N)}$ and  $U^{(N')}$ are separated in the sense that there is a fixed $\delta=\delta(\vp)>0$
such that 
$\forall x\in M_{N\times N'}\quad \|\sum U_j^{(N)} x {U_j^{(N')}}^* \|_2\le n(1-\delta)\|x\|_2$ 
 (see Remark \ref{pet} for more on this).

Motivated by operator theory considerations, it is natural to wonder what happens when  
$N=N'$. We will say that two $n$-tuples $u=(u_j)$ and  $v=(v_j)$ of $N\times N$ unitary matrices are  $\delta$-separated if 
$$\forall x\in M_{N}\quad \|\sum\nolimits_1^n u_j x {v_j}^* \|_2\le n(1-\delta)\|x\|_2.$$
Equivalently this means that
$$\|\sum\nolimits_1^n u_j \otimes \bar  {v_j}\|\le n(1-\delta)$$
where $\bar {v_j}$ denotes the complex conjugate of the matrix $v_j$, and the norm
is the operator norm on $\ell_2^N \otimes \ovl{ \ell_2^N} $.
This can be interpreted in operator space theory as a rough sort of orthogonality related to the
``operator space Hilbert space OH".

Note for example that when \eqref{in0bis} holds then, for any pair of inequivalent irreducible
representations $\pi,\sigma$ on $G$, the $n$-tuples
$(\pi(t_j))$ and $(\sigma(t_j))$ are $\vp$-separated.

Let $U(N)\subset M_N$ denote the  group of unitary matrices. The main result of \S 1 asserts that for any $0<\d<1$ there is a 
  constant  $\beta=\beta_\d>0$   such that for each $0<\vp<1$,
for all sufficiently large integer $n $ (i.e. $n\ge n_0(\vp,\d)$), for any integer $N$
 there is a $\delta$-separated family
 $\{u(t)\mid t\in T\}\subset U(N)^n$ of   $\vp$-quantum expanders such that
$$|T|\ge \exp{\beta nN^2}.$$

Thus we can ``pack"  as many as $m=\exp{\beta nN^2}$  $\delta$-separated 
$\vp$-quantum expanders inside $U(N)^n$. This number $m$ is remarkably large. In fact,
in some sense  it is as large as can be. Indeed, it is known (\cite{Was,HRV}, see also Remark \ref{hrv})
that the maximal $m$ is at most $\exp{\beta' nN^2}$ for some   constant $\beta'$.

In \S 2, we use quantum expanders to investigate the analogue for operator spaces of a well known geometric 
property of Euclidean space: The unit sphere in a Hilbert space is
smooth. Equivalently all its points admit a unique norming functional. In our extension of this,
``norming" will be with respect to the operator space  duality. Moreover,   unicity has to be understood modulo an equivalence relation:
for any $x=(x_j)\in M_N(E)^n$ we define $Orb(x)$ as the set of all $x'$
of the form $x'=(ux_jv)\in M_N(E)^n$ for some $u,v\in U(N)$. Then if $x$ is ``norming" some point,
any $x'\in Orb(x)$ is also `norming" that same point. When $E$ is an operator space
and $x\in M_N(E)$,  we will say that $y\in M_N(E^*)$ $M_N$-norms  $x$ if
$\|\sum x_j \otimes y_j\|= \|x\|_{M_N(E)}  \|y\|_{M_N(E^*)}$. 
We will say that $x$ is $M_N$-smooth in $M_N(E)$ if 
the only points $y$ with $\|y\|_{M_N(E^*)}=1$  that $M_N$-norm $x$ are all in a single orbit in $M_N(E^*)$.
Let us now turn to the case $E=OH_n$. There we show that, if $x\in U(N)^n$ is viewed
as an element of $M_N(\ell_2^n)$, then $x$ is $M_N$-smooth in $M_N(E)$ iff
$x$ is an $\vp$-quantum expander for some $\vp>0$.

More generally, in Lemma \ref{lem7} we prove a more precise quantified version of this: if $x$ is an $\vp$-quantum expander 
and if two points $y,z\in M_N(E^*)$ both  $M_N$-norm $x$ up to some error $\delta$, then the distance
of the orbits $Orb(y)$ and $Orb(z)$ is uniformly small, i.e. majorized by a function $f_\vp(\delta)$ that tends to 
0 when $ \delta\to 0$. Here the distance is meant 
with respect to the renormalized Euclidean norm $y\mapsto (nN)^{-1/2} \|y\|_2 $
for which
any $y\in U(N)^n$ has norm 1 (where $\|.\|_2$ denotes here the norm in $\ell_2(n\times N^2)$).

This also has a geometric application. Consider the following problem
for an $n$-dimensional normed space $E$: Given a constant $C>1$, estimate the minimal number $k=k_E(C)$
of functionals $f_1,\cdots f_k$ in the dual $E^*$ such that   
$$\forall x\in E\quad  \sup_{1\le j\le k} |f_j(x)|  \le  \|x\|\le C \sup_{1\le j\le k} |f_j(x)| . $$
Geometrically this means that  (in the real case)  the symmetric convex body that is the unit ball of $E^*$
is equivalent (up to the factor C) to a  polyhedron with vertices included in $\{\pm f_j\}$ and hence with at most
$2k$ vertices (so its polar, that is equivalent to the unit ball of $E$, has at most $2k$  faces).
For instance, the $n$-dimensional  cube has $2^n$ vertices and $2n$ faces.
When $E$ has (real) dimension $n$ it is well known (see e.g. \cite[p.49-50]{P-v}) that
$$k_E(C)\le (\frac{3C}{C-1})^n.$$ 
For example if $C=2$ we have $k_E(C)\le 6^n.$ This  exponential order of growth in $n$ is optimal
for $E=\ell_2^n$ (or $\ell_p^n$ for $1\le p<\infty$);  
but of course $k_E(C)=n$ for $E=\ell_\infty^n$, and there is important available information and a conjecture (see \cite{Pm})
about conditions on a general sequence $\{E(n)\mid n\ge 1\}$ with $\dim({E(n)})=n$ ensuring that 
$k_{E(n)}\ge \exp{cn}$ for some $c>0$.

We now describe the matricial analogue of $k_E$ that we estimate using quantum expanders.
Let $E$ be an operator space. Fix an integer $N\ge 1$. We denote by $k_E(N,C)$
the smallest $k$ such that there are linear maps $f_j:\ E\to M_N$ ($1\le j\le k$) satisfying
   $$\forall x\in M_N(E)\quad  \sup_{1\le j\le k} \|(Id\otimes f_j)(x)\|_{M_N(M_N)}   \le  \|x\|_{M_N(E)}\le C  \sup_{1\le j\le k} \|(Id\otimes f_j)(x)\|_{M_N(M_N)} . $$
It is not hard to adapt the corresponding Banach space argument to show
that for \emph{any} $n$-dimensional $E$, any $ C>1$ and any $N$ we have $$  k_E(N,C)\le (\frac{3C}{C-1})^{2nN^2}  =\exp 2 \log(\frac{3C}{C-1}) nN^2 .$$
 Using the   "packing" of $\vp$-quantum expanders described above, we can show
 that the operator space version of Hilbert space (i.e. the space $OH$ from \cite{P3})
 satisfies a lower bound of the same order of growth, namely we show for $E=OH_n$
 (see Theorem \ref{os})  
there are numbers $C_1 >1$ , $b>0$  
          such that for any $n$ large enough  and any $N$
         we have
       \begin{equation}\label{33} k_{E}(N,C_1)\ge \exp{ b nN^2}.\end{equation}
        Moreover, this also holds for $E=\ell_1^n$ with its maximal operator space structure
        and for $E=R_n+C_n$ (see Remark \ref{more}).\\
        We also show  (see Theorem \ref{ent})   
          that for any $R>1$ and  for any $n,N$ suitably large
        there is a collection $\{E_t\mid t\in T_1\}$ of $n$-dimensional subspaces of $M_N$
        (each spanned by an $n$-tuple of unitary matrices) 
        with cardinality $\ge  \exp{ \beta_R nN^2}$  such that
        the $cb$-distance $d_{cb}(E_s,E_t)$ of any distinct pair in $T_1$ satisfies
        $$d_{cb}(E_s,E_t)\ge R.$$
        The $cb$-distance $d_{cb}$ is the analogue of the Banach-Mazur distance
        for operator spaces. The preceding 
        shows that the metric entropy of the space of $n$-dimensional  operator spaces
        equipped with the  (so-called) ``distance" $d_{cb}$ is extremely large
        for small distances. This
        can be viewed as a somewhat more quantitative version
        of the non-separability of the space of $n$-dimensional  operator spaces
        first proved in \cite{JP}.  We plan to return to this in a future publication (see \cite{Pbm}).
        
The above \eqref{33} suggests that the class of finite dimensional  operator spaces $E$
such that $\log k_{E}(N,C)/N^2 \to 0$ should be investigated. We call such spaces
matricially subGaussian.

In the forthcoming paper \cite{Psub} we introduce
a  class of operator spaces, that we call ``subexponential", for which the same
Grothendieck type  factorization theorem from  \cite{JP,PS} still holds
(see the recent paper \cite{RV} for simpler proofs of the latter). We also
give there examples of non-exact subexponential operator spaces or $C^*$-algebras.

The definition of ``subexponential" involves the growth of
a sequence of integers $N\mapsto K_E(N,C)$ attached to an operator space $E$ (and a constant $C>1$),
in a way that is similar but seems different from $k_E(N,C)$.
We denote by $K_E(N,C)$
the smallest $K$ such that there is a single (embedding) linear map $f:\ E\to M_K$  satisfying
   $$\forall x\in M_N(E)\quad    \|(Id\otimes f)(x)\|_{M_N(M_K)}   \le  \|x\|_{M_N(E)}\le C  \|(Id\otimes f)(x)\|_{M_N(M_N)} . $$
 Roughly the latter sequence
is bounded iff $E$ is  exact with exactness constant $\le C$ 
(in the sense of \cite[\S 17]{P4}) while it is such that $\log K_E(N,C)/ N\to 0$ iff $E$ is $C$-subexponential.

\n  {\bf Note:}  There is  an obvious upper bound (for a fixed constant $C$)  $K_E(N,C)\le Nk_E(N,C)$,
so the growth of $K_E$ is dominated by that of $k_E$, but we know nothing in the converse direction. Various other questions are mentioned at the end of \S \ref{s3}.

\section{Quantum Expanders}
Fix integers $n,N$.
Throughout this paper we denote by $M_N$ the space of $N \times N$ complex matrices
and by
  $U(N)$   the subset of $N \times N$ unitary matrices.\\
 We identify $M_N$  with the space $B(\ell_2^N)$ of bounded operators on the $N$-dimensional Hilbert space denoted by $\ell_2^N$.
 
 We denote by $tr$ (resp. $\tau_N$) the usual trace (resp. the normalized trace) on $M_N$. 
 Thus $ \tau_N=N^{-1} {\rm tr}$.
 We denote by $S_2^N$   the Hilbert space obtained by equipping
 $M_N$ with the corresponding scalar product. The associated norm is the classical Hilbert-Schmidt norm.
 
 For simplicity we denote by $$H=L_2(\tau_N),$$   i.e.  $H$ is the Hilbert space
obtained by equipping the space $M_N$   with the 
norm
$$\|\xi\|_H= (N^{-1} {\rm tr} (|\xi|^2)^{1/2}=N^{-1/2}\|x\|_{S_2^N} .$$ 
 We denote
$$H_0=\{I\}^\perp \subset H.$$
Throughout this paper, we  consider operators of the form $T=\sum x_j\otimes \bar y_j$,
with $x_j,y_j\in M_N$, that we view as  acting on $\ell_2^N \otimes \overline{\ell_2^N}$.
Identifying as usual  $\ell_2^N \otimes \overline{\ell_2^N}$ with $S_2^N$, we may consider    $T$ as an operator acting on 
$ M_N$ 
 defined by
$$\forall \xi \in M_N
\quad T(\xi)=\sum x_j \xi {y_j}^*,$$
and we then have
\begin{equation}\label{25} \|  \sum x_j\otimes \bar y_j\|=\sup\{ \|\sum x_j \xi y_j^*\|_2\mid \xi\in M_N\  \|\xi\|_2\le 1\} =\sup\{  |\sum {\rm tr}( x_j \xi y_j^* \eta^*)|\mid   \|\xi\|_2\le 1  \|\eta\|_2\le 1\},
\end{equation}
or equivalently
 $\|  \sum x_j\otimes \bar y_j : \ \ell_2^N \otimes \overline{\ell_2^N}\to  \ell_2^N \otimes \overline{\ell_2^N}\|=\|T:\ S_2^N\to S_2^N \|$.
Actually it will be convenient to view $T$ as an operator acting on $H=L_2(\tau_N)$.
We have trivially $$\|T\|_{B(H)}=\|T\|_{B(S_2^N)}.$$

Let $x=(x_j) \in  (M_N)^n$ and $y=(y_j)\in (M_N)^n  $.
Let $Orb(x)$ denote the 2-sided unitary orbit of $x=(x_j)$, i.e.
$$Orb(x)=\{ (ux_jv)\mid u,v\in U(N)\}.$$
We will denote
$$d(x,y)=(\sum\nolimits_j \|x_j-y_j\|^2_{L_2(\tau_N)})^{1/2},$$
and
$$d'(x,y)=\inf \{d( x',y)\mid x' \in Orb(x)\}=\inf \{d( x',y')\mid x' \in Orb(x), y'\in Orb(y)\}. $$
The last equality holds because of the 2-sided unitary invariance of the norm in 
$S_2^N$ or equivalently of $H=L_2(\tau_N)$.

\def\tr{{\rm tr}}

\begin{dfn} Fix $\delta>0$. We will say that $x,y$ in $M_N^n$ are 
$\delta$-separated if $$\|\sum x_j \otimes \bar y_j\|\le (1-\delta) \|\sum x_j \otimes \bar {x_j}\|^{1/2} \|\sum y_j \otimes \bar {y_j}\|^{1/2} .$$ \\
A family of elements is called $\delta$-separated if any two distinct members in it
are $\delta$-separated.
\end{dfn}
Let $x=(x_j)\in M_N^n$ and $y=(y_j)
 \in  M_N^n$ be normalized so that
$ \|\sum x_j \otimes \bar {x_j}\| =\|\sum y_j \otimes \bar {y_j}\|=1$. 
 Equivalently, this definition means
that for any $\xi,\eta\in M_N$ in the unit ball of $S_2^N$ we have
$$|\sum \tr (x_j\xi y_j^*\eta^*)|\le 1-\delta.$$

Using polar decompositions ${\xi}=u|{\xi}|$ and ${\eta}=v|{\eta}|$, 
$|\sum \tr (x_j{\xi}y_j^*{\eta}^*)|=|\sum \tr (x_j u|{\xi}|y_j^*|{\eta}|v^*)|$.
Let $\hat x_j=v^*x_j u$. Equivalently we have  for any $u,v$ unitary
$$|\sum \tr (\hat x_j|{\xi}|y_j^*|{\eta}|)|\le 1-\delta.$$
A fortiori, taking $|{\xi}|=|{\eta}|=N^{-1/2}I$ we find $|\sum \tau_N(   \hat x_j y_j^*)|\le 1-\delta$ and hence
$$d(\hat x, y)^2 \ge 2 \delta $$
and hence taking the inf over $u,v$ unitary, the $\delta$-separation of $x,y$ implies
 \begin{equation}\label{compa} d'( x, y)  \ge (2 \delta )^{1/2}.\end{equation}
In other words, rescaling this to the case when $n^{1/2} x_j, n^{1/2} y_j, \xi,{\eta}$  are all unitary, we have proved:

 \begin{lem} Consider $n$-tuples $x=(x_j)\in U(N)^n$ and $y=(y_j)\in U(N)^n$.
 If $x,y$ are $\delta$-separated then
 $d'( x, y)  \ge (2 \delta n)^{1/2}.$
 \end{lem}

Recall that we denote
$$H_0=\{I\}^\perp .$$

To any $n$-tuple $u=(u_j)\in U(N)^n$  we associate
the operator $ (\sum u_j \otimes \bar u_j)(1-P)$ on $\ell_2^N \otimes \overline{ \ell_2^N}$
where $P$ denotes the $\perp$-projection onto the scalar multiples of $I=\sum e_j\otimes \bar e_j$.
Equivalently, up to the normalization, we will consider
$$T_u:\ H_0\to H_0$$
defined for all $\xi \in H_0$ by
$$T_u(\xi)= \sum u_j\xi u_j^*.$$
We will denote by $$S_\vp=S_\vp(n,N)\subset U(N)^n$$ the set of all   $n$-tuples
$u=(u_j)\in U(N)^n$ such that 
$$\|T_u:\ H_0\to H_0\|\le \vp n.$$
Equivalently, this means $\forall x\in M_N$,   we have
$$\|\sum u_j (x -\tau_N(x)I) u_j^* \|_H\le \vp n  \| x\|_H.$$

Our goal is to prove the following:
 
\begin{thm}\label{goal}  For any $0<\d<1$ there is a 
  constant $ \beta_\d>0$   such that for each $0<\vp<1$   and
for all sufficiently large integer $n$ (i.e. $n\ge n_0$   with $n_0$ depending
on   $\vp$ and $\d$) and for all  $N\ge 1$,  there is a $\delta$-separated subset $$T\subset S_\vp$$  such that
$$|T|\ge \exp{\beta_\d nN^2}.$$
\end{thm}

\begin{rem}\label{mg} Actually, the proof will show that  if we are given 
sets $A_N\subset U(N)^n$ such that $\inf\nolimits_N \P(A_N)\ge \alpha>0$,
then for each $N$ we can find a subset $T$ as above with  $T\subset A_N\cap S_\vp$, but with
 $\beta_\d$ and $n_0$  
now also depending on $\alpha$.
 \end{rem}
 
   \begin{rem}\label{hrv} The order of growth  of our lower bound  $\exp \beta nN^2$ in Theorem\ref{goal}  is roughly optimal  because of the upper bound given explicitly
  in \cite{HRV} (and implicitly in \cite{Was}). The latter upper bound can be proved   as follows.
  Let $m_{\max}$ be the maximal number of a $\delta$-separated family in $U(N)^n$.
 Consider the normed space obtained by equipping   $M(N)^n$ with the norm
 $|||x|||=\|\sum x_j \otimes \bar x_j \|^{1/2}$. Then since its (real) dimension is $2nN^2$,
 by a well known volume argument (\cite[p.49-50]{P-v}) there cannot exist more than $(1+2/\delta')^{2nN^2}$
 elements in its unit ball    at mutual $|||. |||$-distance $\ge \delta'$.
 Note that  $d(x,y)\le  |||x-y|||$ for any pair $x,y$ in  $M(N)^n$. 
  Thus, 
 if $u,v\in U(N)^n$ are $\delta$-separated in the above sense
 then $x=n^{-1/2} u$ and  $y=n^{-1/2} v$ are in the $|||. |||$-unit ball and by \eqref{compa} we have 
 $|||x-y|||\ge (2\delta)^{1/2}$, therefore 
 $$m_{\max} \le (1+\sqrt {2/\delta})^{2nN^2}\le \exp\{ 2\sqrt {2/\delta}\  nN^2\}.$$
  \end{rem}
 \begin{rem}\label{kaz} Let $G$ be a Kazhdan group (see \cite{BHV}) with generators $t_1,\cdots,t_n$, so that there
 is $\delta>0$ such that
 $\|\sum\nolimits_1^n \pi(t_j)\|\le n(1-\delta)$ for any unitary representation
 without any invariant (non zero) vector. Let $\cl I= \cl I(N)$ denote the set of
 $N$-dimensional  irreducible representations $\pi :\ G\to U(N)$. It is known (see \cite{BHV})
 that the latter set is finite and in fact there is a uniform bound on $|\cl I(N)|$ for each $N$.
  For any $\pi \in \cl I$
 we set $$u^\pi_j=\pi(t_j).$$
 Then (here by   $\pi\not= \sigma$  we mean $\pi$ is not equivalent to $\sigma$)
 $$\sup_{\pi\not= \sigma\in \cl I} \|\sum u^\pi_j\otimes \overline{u^\sigma_j} \|\le n(1-\delta),$$
 so that the family $\{u^\pi\mid \pi \in \cl I\}\subset U(N)^n$ is $\delta$-separated in the above sense.
 By the preceding Remark, we know $|\cl I(N)|\le m_{\max}   \le    \exp  c_\delta nN^2$.
 The problem to estimate the maximal possible value of $|\cl I (N)|$ when $N\to \infty$
 (with $\delta$ and $n$ remaining fixed, but $G$ possibly varying)
 is investigated in \cite{MW}:  some special cases are constructed in \cite{MW}
 for which   $|\cl I (N)|$ grows like $\exp c N$, however we feel that Theorem
 \ref{goal} gives evidence that there should exist cases for which
 $|\cl I (N)|$ grows like $\exp c N^2$.
  \end{rem}
  
  \begin{rem} Recall (see \cite[p. 324. Th. 20.1]{P4}) that for any $n$-tuple of unitary operators on any Hilbert space $H$ we have $$\|\sum u_j\otimes  \bar u_j \|\ge 2\sqrt{n-1} .$$
\end{rem}
Note that $2\sqrt{n-1}<n$ for all $n\ge 3$ (so there is also an $0<\vp<1$ such that 
$2\sqrt{n-1} +\vp n<  n$).

   Let $0<\vp<1$. In analogy with Ramanujan graphs (see \cite{Lu}) 
   an $n$-tuple $u=(u_j)\in U(N)^n$ will be called
$\vp$-Ramanujan if 
$$\|T_u:\ H_0\to H_0\|\le 2\sqrt{n-1} +\vp n .$$
We will denote by $$R_\vp=R_\vp(n,N)\subset U(N)^n$$ the set of all   such $n$-tuples.\\
We refer to \cite{Lu,HLW} for more information on expanders and
Ramanujan graphs.

The next result due to Hastings \cite{Ha} has been a crucial inspiration for our work:
 \begin{lem}[Hastings]\label{has} If we equip $U(N)^n$ with its normalized Haar measure $\P$,
 then for each $n$ and $\vp>0$ the set $  R_{\vp}(n,N)$ defined above satisfies
 $$\lim_{N\to \infty}\bb P(R_{\vp}(n,N)) = 1.$$
\end{lem}
This is  best possible in the sense that  Lemma \ref{has} fails if $2\sqrt{n-1}$ is replaced 
(in the definition of $R_{\vp}(n,N)$) by
any smaller number. However,  we do not really need this sharp form of
Lemma  \ref{has} (unless we   insist on making $n_0(\vp)$ as small as possible in Theorem \ref{goal}).
So we give in the appendix a quicker proof of a result that suffices
for our needs (where  $2\sqrt{n-1}$ is replaced  by $C'\sqrt{n}$, $C'$ being a numerical constant)
and which in several respects gives us better estimates than Lemma \ref{has}.

Lemma \ref{lem5} below can be viewed as  a non-commutative variant of results in \cite{P-e}
(see also  \cite{P-s} where the non-commutative case is already considered) in the style of \cite{MP} (see also \cite{E,Pa}). We view this
  as a (weak) sort of non-commutative Sauer lemma, that it might be worthwhile to strengthen.

In the next two lemmas, we equip $U(N)^n$ with the metric $d$ (we also use $d'$), and
for any subset $A\subset U(N)^n$ and 
 any $\vp>0$ we denote by $N(A,d,\vp)$ the smallest number of open $d$-balls 
 of radius $\vp$ with center in $U(N)^n$
 that cover $A$.

 \begin{lem}\label{lem5} Let $a>0$. Let $A\subset U(N)^n$ be a (measurable) subset with $\bb P(A)> a$.
  Then, for any $c<\sqrt 2$, 
  $N(A, d, c \sqrt n) \ge a\exp{Kr nN^2}$
  where $r=(1-c^2/2)^2$ and $K$ is a universal constant.
  \\
 Assuming moreover that $a\ge \exp{-K r nN^2/2}$ (note that there is $n_0(a,r)$ so that this holds for all $n\ge n_0(a,r)$ and all $N\ge1$),
  we find that $N(A, d,  c\sqrt n) \ge  \exp{b nN^2}$
 where $b=Kr/2 $.
 \end{lem}

\begin{proof} Let $\Omega=U(N)^n$. We may clearly assume
(by Haar measure inner regularity) that $A$
is compact. Let ${\cl N}=N(A,d,c\sqrt{n})$.
By definition,   $A$ is included in the union of ${\cl N}$ open balls
 with $d$-radius $c \sqrt n$.
By translation invariance of $d$ and ${\bb P}$, all these balls
have the same ${\bb P}$-measure equal to $F(c)$.
 Therefore
$a<\bb P(A) \le {\cl N}F(c)$ and hence
$$aF(c)^{-1}<    {\cl N} . $$
Thus we need a lower bound for $F(c)^{-1}$.
Let $u$ denote the unit in $U(N)^n$ so that $u_j=1$ for $1\le j\le n$. Using a ball centered at $u$ to compute $F(c)$,
we have
$$F(c)={\bb P}\{ \omega\in U(N)^n\mid \sum\nolimits_1^n {\rm tr}(|\omega_j-1|^2)<c^2nN\}.$$
Since $\sum\nolimits_1^n {\rm tr}(|\omega_j-1|^2)=2Nn-2\sum\nolimits_1^n\Re {\rm tr}(\omega_j)$,
we have
$$F(c)={\bb P}\{ \omega\mid
\sum\nolimits_1^n\Re {\rm tr}(\omega_j)>nN(1-c^2/2)\}.$$
We will now use the known subGaussian property of  $\sum\nolimits_1^n\Re {\rm tr}(\omega_j)$:
there is a universal constant $K$ such that for any $\lambda>0$ we have
 \begin{equation}\label{eq66}{\bb P}\{ \omega\mid
\sum\nolimits_1^n\Re {\rm tr}(\omega_j) >\lambda \}\le  \exp{ -K \lambda^2/n}.\end{equation}
Taking this for granted, let us complete the proof.
Fix $c<\sqrt 2$. Recall $r=(1-c^2/2)^2>0$,
this yields
$$F(c)\le  \exp{ -K nN^2 r}.$$
Thus we conclude that
$$    {\cl N}>a\exp{ Kr nN^2} .$$
Taking $c=1, r=1/4$, the last assertion becomes obvious.

Let us now give a quick   argument  for the known inequality  \eqref{eq66}:
We will denote by $Y^{(N)}$ a random $N\times N$-matrix with i.i.d. 
  complex Gaussian entries with mean zero and second moment equal to $N^{-1/2}$, and we
   denote by $(Y_j^{(N)})$ a sequence of i.i.d. copies of $Y^{(N)}$.
It is well known that the polar decomposition $Y^{(N)}=U |Y^{(N)}|$ is such that
$U$ is uniformly distributed over $U(N)$ and independent of $ |Y^{(N)}|$. Moreover
  there is an absolute constant $\chi>0$ such that $\E |Y^{(N)}|= \chi^{-1} I$. See e.g.
  \cite[p. 80]{MP}. 
  Therefore, we have a conditional expectation operator $\cl E$ (corresponding to integrating
  the modular part) such that
  $\sum \Re {\rm tr}(\omega_j)=\chi {\cl E}(\sum \Re {\rm tr}(Y_j^{(N)})$,
  where $\omega_j$ denotes the unitary part in the polar decomposition of $Y_j^{(N)}$.\\
  Then,
  since $x\mapsto \exp w x$ is convex for any $w>0$, we have the announced subGaussian property
  $$\E \exp w \sum \Re {\rm tr}(\omega_j)\le \E \exp w \chi\sum \Re {\rm tr}(Y_j^{(N)})=\exp (\chi^2 w^2 n/4),$$
  from which follows, by Tchebyshev's inequality, that 
  $\P\{  \sum \Re {\rm tr}(\omega_j)>\lambda\} \le \exp (\chi^2 w^2 n/4-\lambda w ) $ and optimising $w$ so that
  $\lambda=\chi^2 w n/2$ we finally obtain
  $$\P\{  \sum \Re {\rm tr}(\omega_j)>\lambda\} \le \exp (-K\lambda^2/ n),$$
with $K=\chi^{-2}$. The above simple argument follows \cite[ch. 5]{MP}, but, in essence, \eqref{eq66}
can traced back to \cite[Lemma 3]{FTR}.
 \end{proof}

  The next Lemma is a simple covering argument.
   \begin{lem}\label{lem6} Fix $b,c>0$.  Let $A\subset U(N)^n$ be a subset 
  with $N(A,d,c\sqrt{n}) \ge \exp{bnN^2} $. Fix $c'<c$ and $b'<b$. Then there is
  an integer $n_0$ (depending only on
  $b-b'$ and $c-c'$ and independent of $N$) such that if $n\ge n_0$ 
  we have $N(A,d',c'\sqrt{n}) \ge \exp{b'nN^2}$ and  there is a subset $T'\subset A$
  with $|T'|\ge \exp{b'nN^2}$  such that $d'(s,t)\ge c'\sqrt{n}$
  $\forall s\not=t\in T'$.
 \end{lem} \begin{proof}   Fix $\vp>0$. It is well known that there is an $\vp$-net   ${\cl N}_\vp\subset U(N)$  with respect to the operator norm
  with $|{\cl N}_\vp|\le (K/\vp)^{2 N^2}$. Indeed, 
  since the real dimension of $M_N$ is $2N^2$, a classical volume argument
(see e.g. (\cite[p.49-50]{P-v})) produces such a net inside the unit ball of $M_N$. It can then be
adjusted to be inside $U(N)$. See also \cite[p. 175]{Sz} for more delicate estimates.
   For any  $x\in U(N)^n$, we have $d(uxv,u'xv')\le (\|u-u'\|+\|v-v'\|)\sqrt{n}$ for any $u,u',v,v'\in U(N)$.
   Therefore  
  we have $N(Orb(x), d, 2\vp \sqrt{n})\le |{\cl N}_\vp|^2\le \exp{4 N^2\log(K/\vp)}$.
  From this follows immediately that
  $$N(A,d, c'\sqrt{n}+2\vp \sqrt{n})\le  N(A,d',c'\sqrt{n})  \exp{4 N^2\log(K/\vp)}.$$
 Since $c'<c$ we can choose $\vp>0$   so that $c'+2\vp= c$.
 Then by our assumption $N(A,d, c'\sqrt{n}+2\vp \sqrt{n})= N(A,d, c\sqrt{n} )\ge  \exp{bnN^2} $.
  Thus we find
  $$N(A,d',c'\sqrt{n}) \ge \exp{bnN^2} \exp{-4 N^2\log(K/\vp)}.$$
 Since $b-b'>0$ there is clearly an integer $n_0$
  (depending only on $b-b'$ and $\vp=(c-c')/2$) such that  $  4 \log (K/ \vp)<(b-b')n$ for all $n\ge n_0$.
  Thus we obtain 
  $$N(A,d',c'\sqrt{n}) \ge \exp{b'nN^2}.$$
  The last assertion is then clear:   any maximal subset $T'\subset A$ such that $d'(s,t)\ge c'\sqrt{n}$
  $\forall s\not=t\in T'$ must satisfy (by maximality) $N(A,d',c'\sqrt{n})\le |T'|$.
  \end{proof}
  In general, for a pair  $u,v\in U(N)^n$, $\d$-separation is a much stronger condition
  than separation with respect
to the distance $d'$.
The main virtue of the next two Lemmas
is to show that for a pair $u,v\in S_\vp$ with $\vp$ suitably small, the two conditions
become essentially equivalent.
To prove these, we will now crucially use the spectral gap. 
 \begin{lem}\label{lem77} 
Let $0<\vp,\vp'<1$. Let  
  $u=(u_j)\in U(N)^n$   and $v=(v_j)\in M_N^n$ merely such that  $\|\sum v_j \otimes \bar v_j\|\le n$.
Assume $u\in S_\vp$ and also
 \begin{equation}\label{qe-2}d'(u,v)\ge \sqrt{2n(1-\vp')}.\end{equation}
Then
$$\|\sum u_j\otimes \bar v_j\|\le n\left( \vp'^{1/5} (2^{-4/5}+2^{6/5})+2\vp^{1/2}\right).$$
 Moreover, if we assume in addition that $v\in S_\vp$, then the preceding estimate can be improved to
  \begin{equation}\label{qe-3}\|\sum u_j\otimes \bar v_j\|\le n(3\vp'^{1/3}+ 2\vp).\end{equation}
Conversely,  it is easy to show that for any pair  $u,v \in M_N^n$    such that
 $\sum \tau_N(|u_j|^2)=\sum \tau_N(|v_j|^2)=n$ (in particular for any $u,v\in U(N)^n$)
  \begin{equation}\label{eq6+}\| \sum u_j \otimes \bar v_j \|\le n\vp' \end{equation}
 implies 
  \begin{equation}\label{eq6''} d'(u,v)\ge \sqrt{2n(1-\vp') }.\end{equation}
\end{lem} 

\begin{proof} Let $\|v\|^2_{\cl H}=(1/n)\sum_j \tau_N |v_j|^2$.
Note that $\|v\|_{\cl H}\le 1$ and $\|u\|_{\cl H}\le 1$. 
 For any $x\in M_N$, we will denote $\|x\|_{L_2(\tau_N)}=(\tau_N |x|^2)^{1/2}$
 and $\|x\|_{L_1(\tau_N)}=\tau_N |x|$.
 Note for later use that a consequence of Cauchy-Schwarz is
  \begin{equation}\label{qe-1}\forall x,y \in M_N\quad 
  \|xy\|_{L_1(\tau_N)}\le \|x\|_{L_2(\tau_N)}\|y\|_{L_2(\tau_N)}.
  \end{equation}
Recall that 
$$\|\sum u_j\otimes \bar v_j\| =\sup\{ |\sum \tau_N (u_j x v_j^* y^*)| \mid x,y\in B_{L_2(\tau_N)}\}.$$
Let us denote $x.u.y=(xu_jy)_{1\le j\le n}$ and let   $F$ be 
the bilinear form on $M_N \times M_N$
defined by $$F(x,y)= \langle x.u.y, v\rangle_{\cl H}
=(1/n)\sum_j \tau_N (xu_jyv_j^*).$$
Then $$\|(1/n)\sum u_j\otimes \bar v_j\|=\|F:\ L_2(\tau_N) \times L_2(\tau_N)\to \CC\|.$$

We start by  the proof of \eqref{qe-3}, assuming that both $u,v$ belong to $S_\vp$.\\
Firstly we claim that for any $x,y\in M_N$ we have
$$|F(x,y)| \le \tau_N|x|\tau_N|y| +\vp \| (1-P)(|x|) \|_{L_2(\tau_N)}\| (1-P)(|y|) \|_{L_2(\tau_N)} $$ and hence assuming $x,y\in B_{L_2(\tau_N)}$ we have
 \begin{equation}\label{qe0}|F(x,y)| \le \tau_N|x|\tau_N|y| +\vp .\end{equation}
To check this claim
we use polar decompositions $x=U|x|$, $y=V|y|$ and we write
$$F(x,y)= \langle U|x|.u.V|y|, v\rangle_{\cl H}=(1/n)\sum_j \tau_N ([|x|^{1/2}u_jV|y|^{1/2} ] [ |y|^{1/2} v_j^*U|x|^{1/2}]$$ 
$$=\langle |x|^{1/2}.u.V|y|^{1/2} ,|x|^{1/2}U^* v|y|^{1/2}\rangle_{\cl H}. $$
Therefore
 \begin{equation}\label{qe1} |F(x,y)|\le \| |x|^{1/2}.u.V|y|^{1/2}\|_{\cl H} \|  |x|^{1/2}U^* v|y|^{1/2} \|_{\cl H}.\end{equation}
Now we observe that if we denote again
by  $T_u$ the operator acting
 on ${L_2(\tau_N)}$ defined by\\ $T_u(x)= \sum u_jxu_j^* -n \tau_N(x) I$ (equivalently $T_u=(\sum u_j \otimes \bar u_j)(1-P)$)
we have for any $a,b\in M_N$
$$\| a.u.b\|_{\cl H}^2= (1/n) \tau_N ( (n \tau_N(bb^*)+T_u(bb^*))a^*a )=\tau_N(bb^*)\tau_N(a^*a ) + (1/n) \tau_N (T_u(bb^*)a^*a )$$ 
and hence since $\|T_u\|\le \vp n$  and $T_u=(1-P)T_u(1-P)$
$$\| a.u.b\|_{\cl H}^2\le \tau_N(bb^*) \tau_N(a^*a ) +\vp \|(1-P)(bb^*)\|_{L_2(\tau_N)} 
\|(1-P)(a^*a)\|_{L_2(\tau_N)}.$$
This yields
$$\| |x|^{1/2}.u.V|y|^{1/2}\|^2_{\cl H}\le  \tau_N|x|\tau_N|y| +\vp  \| (1-P)(|x|) \|_{L_2(\tau_N)}\| (1-P)(|y|) \|_{L_2(\tau_N)}.$$
A similar bound holds for $\|  |x|^{1/2}U^* v|y|^{1/2} \|_{\cl H}$.
Thus \eqref{qe1} leads to our claim.

Secondly   by \eqref{qe-2}
 for any $U,V\in U(N)$ we have
$$\|U.u.V-v\|^2_{\cl H} \ge n^{-1} d'(u,v)^2\ge  2(1-\vp')$$
and hence
$$\Re \langle U.u.V, v\rangle_{\cl H} \le \vp'.$$
Recall that the unit ball of $M_N$ is the closed convex hull
of $U(N)$. 
Thus  we have
$$\|F:\ M_N \times M_N\to \CC\| \le \vp'.$$
 \def\la{\lambda}
Let us assume $x,y\in B_{L_2(\tau_N)}$.
Let  $p$ (resp. $q$) denote the spectral projection
of $|x|$  (resp. $|y|$) corresponding to the spectral set $\{|x|\le \lambda\}$
(resp. $\{|y|\le \lambda\}$). Note that by Tchebyshev's inequality
we have $\tau_N(1-p)
\le 1/\lambda^2$ (resp. $\tau_N(1-q)
\le 1/\lambda^2$).
Let $x'=(1-p)|x|$ and $y'=(1-q)|y|$.
By   \eqref{qe-1} (since  $ \|1-p\|_{L_2(\tau_N)}
\le  \lambda^{-1/2}$)    we have
$$ \tau_N|x'|\le  \la^{-1}.$$
Similarly
$$ \tau_N|y'|\le  \la^{-1}.$$
We now write
\begin{equation}\label{qe3}F(|x|,|y|)=F(p|x|+x',q|y|+y')=F(p|x|,q|y|)+F(x',|y|)+F(p|x|,y'). \end{equation}
By \eqref{qe0}  we have $$|F(x',|y|)|\le \tau_N | x'|  \tau_N | y|+\vp \le  \la^{-1}+\vp $$
and similarly
$$|F(p|x|,y')|\le  \la^{-1}+\vp . $$
 Thus we deduce from \eqref{qe3}
 $$|F(|x|,|y|)|\le     \vp'  \la^2 +2(\la^{-1}+\vp) .$$
 Choosing $\la= (\vp')^{-1/3}$ to minimize over $\la>0$ yields the  upper bound $3\vp'^{1/3}+ 2\vp$, when restricting to $x,y\ge 0$.
 Since our assumptions on the pair $u,v$ are shared by the pair $UuV,v$
 for any $U,V\in U(N)$, we may apply the polar decompositions $x=U|x|$
 and $y=V|y|$ to deduce the same upper bound for an arbitrary pair $x,y$.
 Thus we obtain \eqref{qe-3}. 
 
 We now turn to \eqref{qe-2}. There we assume only $u\in S_\vp$
 and $\|\sum v_j\otimes \bar v_j\|\le n$.  
 Then (since we still have $\|  |x|^{1/2}U^* v|y|^{1/2} \|_{\cl H}\le 1$)
 \eqref{qe0} can be replaced by
 \begin{equation}\label{qe00}|F(x,y)| \le (\tau_N|x|\tau_N|y| +\vp)^{1/2} ,\end{equation}
 and the preceding reasoning leads to
 $$|F(x,y)|\le     \vp'  \la^2 +2(\la^{-1}+\vp)^{1/2}\le \vp'  \la^2  +2\la^{-1/2}+2\vp^{1/2} .$$
 Choosing $\la= (2\vp')^{-2/5}  $ to minimize, we obtain the announced upper bound
 $\vp'^{1/5} (2^{-4/5}+2^{6/5})+2\vp^{1/2} ,$ thus completing the proof of \eqref{qe-2}.
 
 The converse implication \eqref{eq6+} $\Rightarrow$ \eqref{eq6''} is obvious:
 Indeed, for any $U,V\in U(N)$ \eqref{eq6+}  implies 
 $  |\sum \tau_N(Uu_jV v_j^*)| \le n\vp'$ and hence
 since we assume $\sum  \tau_N(u^*_ju_j)=\sum  \tau_N(v^*_jv_j)=n$ we have
  $$d(U.u.V,t)^2=2n-2\Re \sum \tau_N(Uu_jV v_j^*)\ge 2n(1-\vp')$$
 and taking the infimum over $U,V\in U(N)$ we obtain  \eqref{eq6''}.
\end{proof}  
    
  \begin{proof}[Proof of Theorem \ref{goal}] 
  Our original proof was based on Hastings's Lemma \ref{has}, but
  the current proof, based instead on \eqref{a44} allows for
  more uniformity with respect to $N$. Note however that
  if we could prove a sharper form of \eqref{qe-3} (e.g. with $\vp$ in place of $2\vp$)
  then Lemma \ref{has} would allow us to cover values of $n$ as small as $n=3$,
  for all $N$ large enough (while using \eqref{a44}   requires $C' n^{-1/2}<n$).
  
   By \eqref{a44} there is a constant $C'$ such that 
  such that 
  $$\forall N\ge 1\quad \E \|\sum\nolimits_1^n   U_j \otimes \bar U_j (1-P)\| \le   C'\sqrt n. $$
  where $\E$ is with respect to the normalized Haar measure on $U(N)^n$.
By  Tchebyshev's inequality,  this implies
    $\P(S_\vp)>1/2$   for all $N\ge 1$,
  assuming only that $n> n_0(\vp)$ for a suitably adjusted value of $n_0(\vp)$ (say
  we require $\vp^{-1}  C' n^{-1/2}<1/2$). We will now apply  
    Lemmas \ref{lem5} and \ref{lem6} to the subset $A=S_\vp$.\\
   Fix $0<\vp,\d<1$.  Let $0<\vp'<1$ be such that
  $3\vp'^{1/3}=(1-\d)/2$ and let $\vp_0$ be such that
  $2\vp_0=(1-\d)/2$.
  Let $c'=\sqrt{2(1-\vp')}$ and $c=\sqrt{2(1-\vp'/2)}$
  so that we have $c'<c<\sqrt 2$.\\
  Assume $\vp\le\vp_0$.
  Let $r=\vp'^2$ and  $b=K\vp'^2/2$ and say $b'=b/2$
  so that $b-b'$ and $c-c'$ depend only on $\vp'$
  (or equivalently on $\d$).
  For $n\ge n_0(\vp,\vp')$, Lemmas \ref{lem5} and \ref{lem6}   give us a subset $T\subset S_\vp$
  such that $|T|\ge \exp K\vp'^2 nN^2/4$
  and such that $d'(s,t)\ge \sqrt{2(1-\vp')}$ for all $s\not=t$.
  Then  Lemma \ref{lem7} (specifically \eqref{qe-3}) gives us
  that $s,t$ are $\d$-separated since $\vp\le\vp_0$ and our choice
  of $\vp_0,\vp'$ is adjusted so that $3\vp'^{1/3}+2\vp_0=1-\d$.
  This completes the proof for any $\vp\le\vp_0$ and in particular
  for $\vp=\vp_0$. The remaining case $\vp_0<\vp<1$ then follows automatically
  since   $S_{\vp_0}\subset S_\vp $ if $\vp_0<\vp$. \end{proof}

One defect of Lemma  \ref{lem77} is that when $\vp'$ is close to 1, its conclusion
is void (however small $\vp$ can be). This is corrected by the next Lemma the main interest of which
is the case when $\d$  and $f_\vp(\d)$ are small.

\begin{lem}\label{lem7} Fix $0<\vp <1$. There is a positive 
  function $\delta  \mapsto f_\vp(\delta )$ defined  for $0<\d<1$
  and such that $  f_\vp(\delta )=O_\vp( \d^{1/4})$ when $\d \to 0$
  satisfying the following property:\\
 Consider $u=(u_j)\in S_\vp \subset U(N)^n$ and 
 and   $v=(v_j)\in M_N^n$  such that $\|\sum v_j \otimes \bar v_j \|\le n$.
 The condition
 \begin{equation}\label{eq6} d'(u,v)\ge f_\vp(\delta ) \sqrt n\end{equation}
 implies
 \begin{equation}\label{eq6'}\| \sum u_j \otimes \bar v_j \|\le n(1-\delta) .\end{equation}
 
\end{lem} \begin{proof} Assume by contradiction that $\| \sum u_j \otimes \bar v_j \|> n(1-\delta) $.
Then there are $\xi,\eta$ in the unit sphere of $H=L_2(\tau_N)$ such that
$$ \Re \tau_N(\sum u_j \xi v_j^* \eta^*)>n (1-\delta) .$$
 Let $\xi=U| \xi|$ and $\eta =V|\eta|$ be their polar decompositions, and let $w_j=V^* u_j U$ so that we can write
  \begin{equation}\label{eq7}\Re \tau_N(\sum w_j |\xi| v_j^* |\eta|)>n (1-\delta) .\end{equation}
 Recall $T_u(\xi)= \sum u_j\xi u_j^*.$ Note that since $U\otimes \bar U$ and  $V\otimes \bar V$ preserve $I$ (and hence $I^\perp$), we have $\|T_w\|=\|T_u\|$. Therefore $w\in S_\vp$. Using the scalar product in $H$  we have
 $$ \Re  \sum \langle |\eta|^{1/2} w_j |\xi|^{1/2},  |\eta|^{1/2} v_j |\xi|^{1/2}   \rangle>n (1-\delta) ,$$
 and by Cauchy-Schwarz 
 $$( \sum  \| |\eta|^{1/2} w_j |\xi|^{1/2}\|_H^2)^{1/2}\ ( \sum  \| |\eta|^{1/2} v_j |\xi|^{1/2}\|_H^2)^{1/2}>n (1-\delta) .$$
 Note that since $\|\sum v_j \otimes \bar v_j \|\le n$ we have
 $\langle (\sum v_j \otimes \bar v_j)|\xi|,|\eta| \rangle 
 =  \sum  \| |\eta|^{1/2} v_j |\xi|^{1/2}\|_H^2  \le   n$,
  and similarly with $w_j$ in place of $v_j$.
  Thus the last inequality implies a fortiori
  \begin{equation}\label{eq8}\langle (\sum w_j \otimes \bar w_j)|\xi|,|\eta|\rangle >n (1-\delta)^2 ,
  \end{equation}
 and the same with $v_j$ in place of $w_j$.\\
  Let $e=(1-P)|\xi|$ and $d=(1-P)|\eta|$. Recall that $P(|\xi|)=\tau_N(|\xi|) I$
  and $P(|\eta|)=\tau_N(|\eta|) I$, and $|\xi|,|\eta|$ are unit vectors, so that $\tau_N(|\xi|)=(1-\|e\|_H^2)^{1/2}$
  and $\tau_N(|\eta|)=(1-\|d\|_H^2)^{1/2}$. 
  \def\o{\omega}
  Let $\o=\|e\|_H \|d\|_H$. By Cauchy-Schwarz we have $\o+(1-\|e\|_H^2)^{1/2}(1-\|d\|_H^2)^{1/2}\le 1$ and hence $(1-\|e\|_H^2)^{1/2}(1-\|d\|_H^2)^{1/2}\le 1-\o$.
 Since $w\in S_\vp$, we have
  $$\langle (\sum w_j \otimes \bar w_j)|\xi|,|\eta|\rangle=\langle T_w e, d\rangle + n \tau_N(|\xi|) \tau_N(|\eta|) \le  \vp n \o +n (1-\o).  $$
 Thus, \eqref{eq8} yields
  \begin{equation}\label{qerev2}\o\le (1-\vp)^{-1} (2\delta-\d^2).\end{equation}
 Moreover, \eqref{eq8}  implies
\begin{equation}\label{qerev4}  n^{-1}  \sum \| w_j |\xi|w_j ^*  - |\eta|  \|_H^2= 
 2 -2n^{-1}  \langle (\sum w_j \otimes \bar w_j)|\xi|,|\eta|\rangle
 <2(2\d-\d^2) .\end{equation}
 But since $(I-P)(w_j |\xi|w_j ^*- |\eta|)=w_jew_j ^*-d$ for each $j$,  we have 
 $$|\| e\|_H-\|d \|_H| =|\| w_jew_j ^*\|_H-\|d \|_H |\le \| w_jew_j ^*-d \|_H\le \| w_j |\xi|w_j ^*  - |\eta|  \|_H$$
so  that \eqref{qerev4}  implies
  $(\| e\|_H-\|d \|_H)^2< 2(2\d-\d^2)$. Therefore  
  \begin{equation}\label{qerev1}\| e\|_H^2+ \| d\|_H^2< 2\o + 2(2\d-\d^2)\le 2(2\d-\d^2)((1-\vp)^{-1}+1).\end{equation}
  We can write
   $$     w_j |\xi| v_j^* |\eta|=
     w_j P(|\xi|) v_j^* P(|\eta|)+
       w_j e v_j^* P(|\eta|)
      +    w_j P(|\xi|) v_j^* d
      +   w_j e v_j^* d$$
      and we find
      $$   n(1-\d)< \Re\tau_N(\sum w_j |\xi| v_j^* |\eta|)$$
      $$\le    \Re\tau_N(\sum w_j  v_j^*)\tau_N(|\xi|)\tau_N(|\eta|)
      + n\|e\|_H \tau_N(|\eta|) +
      n\tau_N(|\xi|) \|d\|_H +n \|e\|_H \|d\|_H
       $$
       and a fortiori 
      $$n(1-\d)<   \Re\tau_N(\sum w_j  v_j^*) \tau_N(|\xi|)\tau_N(|\eta|)+n\|e\|_H +n\|d\|_H+n\|e\|_H\|d\|_H .$$
      Note that $ n\|e\|_H +n\|d\|_H+n\|e\|_H\|d\|_H\le n2^{1/2} (\|e\|_H^2 +\|d\|_H^2)^{1/2}+n\|e\|_H\|d\|_H$.
      By \eqref{qerev1} and \eqref{qerev2}, we obtain
      $$  \tau_N(|\xi|)\tau_N(|\eta|) \Re\tau_N(\sum w_j  v_j^*) > n(1- \theta)$$
      with $$\theta=  \d+ 2^{1/2} \left(2(2\d-\d^2)((1-\vp)^{-1}+1)\right)^{1/2} + (1-\vp)^{-1} (2\delta-\d^2).$$
      Note that $\theta$ is $O(\d^{1/4})$ when $\d\to 0$ and  there is clearly some $\d_\vp>0$ such that $0<\theta<1$ for all $\d\le \d_\vp$. Thus,
       assuming   $0<\d\le \d_\vp$ we have $1-\theta>0$  and  we find
      $$\Re\tau_N(\sum w_j  v_j^*)  > n(1- \theta).$$
      This is  a lower bound for  the real part of the scalar product in $\ell_2^n(H)$ of  $w=(w_j )$ and  $v=(v_j )$ 
 which are both in the   ball of radius $\sqrt n$. Therefore we deduce from this
 $$ d(w,v)^2\le 2n-2\Re\tau_N(\sum w_j  v_j^*)<2n \theta.$$
 If we now set 
  $f_\vp(\delta)=(2 \theta)^{1/2}$  for all $\d\le \d_\vp$, and  $f_\vp(\delta)=3$ (say) for  $\d_\vp<\d< 1$, we have
  in any case $ d(w,v)<f_\vp(\delta) \sqrt{n}$.\\ Thus we have proved that  
 $\| \sum u_j \otimes \bar v_j \|> n(1-\delta) $ implies
  $d'(u,v)<f_\vp(\delta)\sqrt{n}$. 
   This is equivalent to the fact that \eqref{eq6}  implies
  \eqref{eq6'}, and moreover for each $0<\vp<1$
  there is a constant $c_\vp>0$ such that
  for any $0<\d<1$ we have  $f_\vp(\delta) \le c_\vp \d^{1/4}$.
 \end{proof}

       \begin{rem}\label{pet} 
       Let   $f_\vp(\d)$ be any function such that \eqref{eq6}  $\Rightarrow$  \eqref{eq6'}.
       Then, Lemma \ref{lem7} has the following consequence: Assume $u=(u_j)\in S_\vp$. Then for any $v=(v_j)\in U(k)^n$ with $k \le  (1-f_\vp(\delta)^2)N$
 we have
 $$\|\sum u_j \otimes \bar v_j  \| \le  n(1-\delta).$$
 Indeed, if  $\|\sum u_j \otimes \bar v_j\|> n(1-\delta)$, 
 then we set $v'_j=   v_j\oplus 0 \in M_N$
 so that  $\|\sum v'_j \otimes \bar v'_j\|  \le n $, and also
 $\|\sum u_j \otimes \bar v'_j\|> n(1-\delta) $. By Lemma \ref{lem7}, 
 it follows that $d'(u,v')<f_\vp(\delta) \sqrt n$. But since $|\langle u'_j, v'_j\rangle | \le k/N$
 for any $u'_j\in U(N)$, we have $d(u',v')^2=n+n(k/N) -2\sum \langle u'_j, v'_j\rangle  \ge  n(1-k/N)$,
 and hence $d'(u,v')\ge \sqrt {n}(1-k/N)^{1/2}$, which leads to $(1-k/N)^{1/2} <f_\vp(\delta) $. This contradiction concludes the proof.
\end{rem}
\section{Application to Operator Spaces}

We start with a specific notation. 
Let $u:\ E \to F$ be a linear map between operator spaces.
We denote for any  given  $N\ge 1$
$$u_N =Id\otimes u :\ M_N(E)\to M_N(F).$$
Moreover, if $E,F$ are two operator spaces that are isomorphic as Banach spaces, we set
$$d_N(E,F)=\inf \{ \|u_N\| \|(u^{-1})_N\| \}$$
where the inf runs over all the isomorphisms $u:\ E \to F$.
We set $d_N(E,F)=\infty$ if $E,F$ are not isomorphic.\\
Recall that $$\|u\|_{cb}=\sup\nolimits_{N\ge 1} \|u_N\|.$$
Recall also that, if $E,F$ are completely isomorphic, we set
$$d_{cb}(E,F)=\inf \{ \|u\|_{cb} \|u^{-1}\|_{cb} \}$$
where the inf runs over all the complete isomorphisms $u:\ E \to F$.\\
When $E,F$ are both $n$-dimensional, a compactness argument shows that
$$d_{cb}(E,F)=\sup\nolimits_{N\ge 1} d_N(E,F).$$

We will apply the preceding to $M_N$-spaces. When $N=1$, the latter coincide with the usual Banach spaces. When $N>1$, roughly the complex scalars are replaced by $M_N$.

Let $(A_i)_{i\in I}$ be a family of von Neumann or $C^*$-algebras.  Let
$Y=\oplus_{i\in I} A_i$ denote their direct sum. This can be described as the algebra of bounded families $(a_i)_{i\in I}$ with $a_i\in A_i$ for all $i\in I$,  equipped with the norm $\|a\|=  \sup\nolimits_{i\in I} \|a_i\|$.
We will concentrate on the case when $A_i=M_N$ for all $i\in I$.  In that case, following the Banach space tradition,  we denote 
the space $Y=\oplus_{i\in I} A_i$ by $\ell_\infty(I;M_N)$.

\begin{dfn} An operator space $X$ is called an $M_N$-space if, for some set $I$,  it can be embedded
completely isometrically in  $\ell_\infty(I;M_N)$.
\end{dfn}

Our main interest will be to try to understand for which spaces the cardinality of $I$ is unusually small.

To place things in perspective, we recall that for any (complex) Banach space $X$
there is an isometric embedding  $J:\ X \to \ell_\infty(I;\CC)$ defined by
$(J x)(\phi)=\phi(x)$. Here $I$ is the unit ball, denoted by $B_{X^*}$, of the space $X^*$.\\

In analogy with this, for any $M_N$-space there is a canonical completely isometric embedding  $\hat J:\ X \to \ell_\infty(\hat I;M_N)$ defined again by $(J x)(\phi)=\phi(x)$, but with $\hat I=B_{CB(X,M_N)}$ in place of $B_{X^*}$. The space $\ell_\infty(\hat I;M_N)$ can alternatively be described 
as $\oplus_{i\in \hat I} Z_i$ with $Z_i=M_N$ for all $i\in \hat I$. 

Just like operator spaces, $M_N$-spaces enjoy a nice duality theory (see \cite{Leh,OR} for more information). Indeed, by Roger Smith's lemma, we have $\|u\|_{cb}=  \|u_{N}\|$ for
any $u$ with values in an $M_N$-space  (see e.g. \cite[p. 26]{P4}), and
 $M_N$-spaces are characterized among operator spaces
by this property.
The following reformulation of  Smith's Lemma  is  useful.
\begin{lem}\label{rs} Fix an integer $N\ge 1$. Let $E\subset B(H)$ be a finite dimensional operator space and let $c\ge 1$ be a constant. The following properties are equivalent.\\
{\rm (i)} For any operator space $F$ and any $u:\ F\to E$ we have $\|u\|_{cb}\le c\|u_{N}\|$.\\
{\rm (ii)} There is an  $M_N$-space such that $d_{cb}(E,\hat E)\le c$.\\
{\rm (iii)} Let ${\cl C}$ be the class of all (compression) mappings $v:\ E \to B(H',H'')$ 
of the form $x\mapsto P_{H'' }x_{|H'}$ where $H',H''$ are arbitrary  subspaces of $H$ of dimension at most $N$.
Let $\hat J: E \to \oplus_{v\in \cl C} Z_v$ with $Z_v=B(H',H'')$ be defined by $\hat J(x)=\oplus_{v\in \cl C}   v(x)$, and let $\hat E= \hat J(E)$. Then $d_{cb}(E,\hat E)\le c$.
\end{lem}
\begin{proof}  {\rm (ii)} $\Rightarrow$ {\rm (i)}  follows from Roger Smith's lemma and  {\rm (iii)} $\Rightarrow$ {\rm (ii)} is trivial.
Conversely, if  {\rm (i)} holds, let $\hat E$ be the $M_N$-space obtained using the embedding
$\hat J:\ E \to \oplus_{v\in \cl C} Z_v$ appearing in {\rm (iii)}. Obviously
$\|E \to \hat E\|_{cb}\le 1$.
Let us denote by $u:\  \hat E \to   E$ the inverse mapping. 
A simple verification shows that $\|u_N\|=1$ and hence {\rm (i)}  implies $\|u\|_{cb}\le c$. In other words
{\rm (i)} $\Rightarrow$ {\rm (iii)}. 
\end{proof} 
Therefore, when $X$ is an $M_N$-space, the knowledge of the space $M_N(X)$ determines that of $M_n(X)$ for all $n>N$, and hence the whole operator space structure of $X$.

Given a general operator space $X\subset B(H)$, by restricting   to $M_N(X)$
(and ``forgetting" $M_n(X)$ for  $n>N$),
we obtain an $M_N$-space $M_N$-isometric to $X$. We will
say that the latter $M_N$-space is induced by $X$.\\ 
Conversely,
given an $M_N$-space $X$ there is a minimal and a maximal operator space structure  on $X$
inducing the same $M_N$-space. When $N=1$, we recover the Blecher-Paulsen
theory of minimal and maximal operator spaces associated
to Banach spaces, see  \cite{Leh,OR} for more on this.

Let $E$ be a finite dimensional operator space. For each integer $N$, 
let $E[N]$ denote the induced $M_N$-space.
Then it is easy to check that $E$ can be identified (completely isometrically) with the ultraproduct of $\{  E[N]   \}$ relative
to any free ultraproduct on $\NN$. Thus the operator space structure  of $E$ can be encoded
by the sequence of $M_N$-spaces $\{  E[N] \mid N\ge 1  \}$. Note that  $E[N]$ is induced by 
$E[N+1]$ for any $N$, so that one could picture the set of $n$-dimensional  operator spaces
as infinite branches of trees where the $N$-th node consists of an $M_N$-space,
and any node is induced by any successor.

We can associate to each $M_N$-space a dual one $ X^\dagger$, isometric to the 
operator space dual $X^*$, but 
defined by
$$\forall n\in \NN\quad \forall y\in M_n(X^\dagger) \quad \|y\|_{M_n(X^\dagger)}=\sup\nolimits_{f \in M_N(X)} \|(I\otimes f)(y) \|_{M_n(M_N)},$$
where we view $M_N(X)$ as a subset of   $CB(X^*,M_N)$  in the usual way.
In other words we have a completely isometric embedding $J_\dagger :\ X^\dagger \to \ell_\infty(I;M_N)$
defined by $$J_\dagger (z)=\oplus_{f \in M_N(X)} f(z)=\oplus_{f \in M_N(X)} [f_{ij}(z)].$$

Just like for operator spaces, there is a notion of ``Hilbert space" for $M_N$-spaces.
We will denote it by $OH(n,N)$. The latter can be defined as follows.
First we have an analogue of the Cauchy-Schwarz inequality due to Haagerup,
as follows: $\forall x=(x_j) \in M_N^n, \forall y=(y_j)\in M_N^n$
\begin{equation}\label{eqcsh}
\| \sum x_j\otimes \bar y_j\|\le \|\sum  x_j\otimes \bar x_j\|^{1/2}    \|\sum  y_j\otimes \bar y_j\|^{1/2} .
\end{equation}
 Fix $N$. 
Let $S(n,N)$ (resp. $B(n,N)$) denote the set of $n$-tuples $x=(x_j)$ in $M_N$ such that
$\|\sum x_j \otimes \bar {x_j}\|= 1$ (resp. $\|\sum x_j \otimes \bar {x_j}\|\le 1$).
Then $S(n,N)$ (resp. $B(n,N)$)  is the analogue of the unit sphere (resp. ball)
in the $M_N$-space  $OH(n,N)$. The space $X=OH(n,N)$ is 
isometric to $\ell_2^n$, with its orthonormal basis $(e_j)$,  and embedded into
$\ell_\infty(I;M_N)$ with $I=B(n,N)$ (we could also take $I=S(n,N)$).
 The embedding
$J_{oh}: OH(n,N)\to   \ell_\infty(I;M_N) $ 
is defined by
$$ \forall j=1,\cdots,n\quad J_{oh} (e_j)= \oplus_{x\in B(n,N)} x_j.$$

 The latter is the analogue of $n$-dimensional Hilbert space
among $M_N$-spaces, and indeed when $N=1$ we recover the $n$-dimensional Hilbert space.

\begin{dfn}
Let $E$ be an operator space with basis $(e_j)$. Let $\xi_j$ be the biorthogonal basis of $E^*$.
Let $x=\sum x_j\otimes e_j\in M_N(E)$ and $y=\sum y_j\otimes \xi_j \in M_N(E^*)$. Assuming $x\not=0$ and $y\not=0$, we say that $y$ $M_N$-norms $x$ (with respect to $M_N(E)$)
  if 
$$\| \sum x_j\otimes y_j\|= \|x\|_{M_N(E)}  \|y\|_{M_N(E^*)}.$$
In the particular case when $E=OH_n$, we slightly modify this (since $E^*= \bar E$): Given $x,y\in M_N(OH_n)$, we say that $y$ $M_N$-norms $x$ if
$$\| \sum x_j\otimes \bar y_j\|= \|\sum  x_j\otimes \bar x_j\|^{1/2}    \|\sum  y_j\otimes \bar y_j\|^{1/2} .$$
\end{dfn}

Let $x \in M_N(E)$. For $a,b\in  M_N$ we denote by $axb$ the matrix product
(i.e. $(a\otimes 1) x(b\otimes 1)$ in tensor product notation using $M_N(E)=M_N\otimes E$).
We denote $$Orb(x)=\{ uxv \in  M_N(E)\mid u,v\in U(N)  \}.$$

Note that if  $y\in M_N(E^*)$ $M_N$-norms    $x$   then the same is true for any $y'\in Orb(y)\subset M_N(E^*)$. Actually, any   $y'\in Orb(y)$ $M_N$-norms any $x'\in Orb(x)$.

\begin{dfn} We say that  $x \in M_N(E)$ is an $M_N$-smooth point of $M_N(E)$  
if the set of points $y$ in the unit sphere of $ M_N(E^*)$ that $M_N$-norm $x$  
is reduced to a single orbit.
\end{dfn}

The following simple Proposition explains the direction we will be taking next.

\begin{pro} Let $x,y\in M_N(OH_n)$. Assume  $$x=(x_j)\in U(N)^n  \ {\rm and} \   
\|T_x:\ H_0\to H_0\|<n,$$ where $T_x=\sum x_j\otimes \bar x_j(1-P)$ (i.e. $T_x$ has a spectral gap at $n$).\\
Then $y$ norms $x$ with respect to $M_N(OH_n)$ iff
$y$ is a multiple of an element of $Orb(x)$, i.e. iff there are $\lambda >0$
and $u,v\in U(N)$ such that $y_j=\lambda vx_ju$ for all $1\le j\le n$.
\end{pro}
\begin{proof} Recall that  whenever the $x_j$'s are finite dimensional unitaries we have 
$\|\sum  x_j\otimes \bar x_j\|^{1/2}  =\sqrt n$.
 Assume $y$ is a multiple of an element of $Orb(x)$,
i.e. $y_j=\lambda vx_ju$ for some non zero scalar $\lambda$ (that may as well be taken positive if we wish).
Then $\| \sum x_j\otimes y_j\|=|\lambda| n$,  $\|x\|_{M_N(OH_n)}=\sqrt n$  and $\|y\|_{M_N(OH_n)}=|\lambda|\sqrt n$, so
indeed $y$ norms $x$.\\
Conversely, assume that $y$ norms $x$. Multiplying $y$ by a scalar we may assume that
$\|y\|_{M_N(OH_n)}= \sqrt n$, and $\| \sum x_j\otimes \bar y_j\|= \|\sum  x_j\otimes \bar x_j\|^{1/2}\sqrt n=n$.
Let $\xi,\eta$ in the unit sphere of $H=L_2(\tau_n)$ such that 
$$\sum \tau_N( x_j \xi y_j^* \eta^*) =n.$$
Let $\xi=u| \xi|$ and $\eta =v|\eta|$ be the polar decompositions, and let $x'_j=v^* x_j u$.
Using the trace property, this can be rewritten  using the scalar product in $H$ as:
$$\sum\langle   (|\eta|^{1/2} x'_j| \xi|^{1/2}) ,   (|\eta|^{1/2} y_j| \xi|^{1/2}) \rangle =n,$$ 
and hence since $n^{-1/2}(|\eta|^{1/2} x'_j| \xi|^{1/2}) ,   n^{-1/2}(|\eta|^{1/2} y_j| \xi|^{1/2})$ are both in the unit ball
of the (smooth!) Hilbert space  $\ell_2^n(H)$, they must coincide. Moreover they both must be on the unit sphere.
Therefore $\sum \| |\eta|^{1/2} x'_j| \xi|^{1/2}\|^2_H=n$. Equivalently 
$\sum \tau_N( x'_j |\xi| {x'_j}^* |\eta|) =n.$ But we have obviously $\|T_{x'}:\ H_0\to H_0\|=\|T_{x}:\ H_0\to H_0\|<n$.
Therefore $|\xi|$ and $  |\eta|$ must be multiples of $I$, so that by our normalization
we have $|\xi|= |\eta|=I$, and we conclude that $y=x'$.
\end{proof}

In other words, the preceding Proposition shows that quantum expanders
constitute $M_N$-smooth points of $M_N(OH_n)$:

\begin{cor} 
  Assume  $x=(x_j)\in U(N)^n$. Then     
$x=\sum x_j\otimes e_j$ is an $M_N$-smooth point in $M_N(OH_n)$
 iff $\|T_x:\ H_0\to H_0\|<n$.
 \end{cor}
\begin{proof} The ``if part" follows from the preceding statement.
Conversely, we claim that if $\|T_x:\ H_0\to H_0\|=n$ then
  $x$ is not an $M_N$-smooth point in $M_N(OH_n)$. 
  Since this claim is unchanged if we replace $x$ by any $x'$ in $Orb(x)$, we may
  assume that $x_1=1$.
Then if
$\|T_x:\ H_0\to H_0\|=n$,  there is $0\not= \xi\in H_0$ 
such that $\|T_x(\xi)\|= n\|\xi\|$, and hence (by the uniform convexity of Hilbert space)
$x_j \xi x_j^*=x_1 \xi x_1^*=\xi $ for all $j$. This implies that the commutant of $\{x_j\}$ is not reduced
to the scalars, and hence  in a suitable basis 
$x_j=x^1_j\oplus x_j^2\in M_{N_1}\oplus M_{N_2}$ for some $ N_1,N_2\ge 1$   with $N_1+N_2= N$.
Then the choice of $y_j= x^1_j\oplus 0$ produces $y\in M_N^n$ not in $Orb(x)$
and such that $\|\sum x_j\otimes \bar y_j\| =n$. Thus $x$ is not an $M_N$-smooth point in $M_N(OH_n)$,
proving our claim.
 \end{proof}
\begin{rem}\label{rem} Let $E$ be any $n$-dimensional operator space with a basis $(e_j)$.
 Assume that for any $u=(u_j)\in U(N)^n$ we have
 $\|\sum u_j\otimes e_j\|_{M_N(E)}=\sqrt n$ and also that 
 $\|\sum a_j\otimes e_j\|_{M_N(E)}\le \|\sum a_j\otimes \bar a_j\|^{1/2}$ for al $a=(a_j)\in M_N^n$.
Then,   by the same proof, for any $x=(x_j)\in U(N)^n$ such that     
$\|T_x:\ H_0\to H_0\|<n$ as above,  the point
$x=\sum x_j\otimes e_j$ is an $M_N$-smooth point in $M_N(E)$. Indeed, any $y$ in the unit ball of
$M_N(E^*)$ that $M_N$-norms $x$ with respect to $M_N(E)$ is a fortiori in in the unit ball of
$M_N(OH_n)$.
 \end{rem}

Lemma \ref{lem7} above can be viewed as  a refinement of this: assuming $\|T_x:\ H_0\to H_0\|<\vp n$
we have a certain form of ``uniform smoothness" of $OH_n$ at $x$, the points that almost  $M_N$-norm $x$ up to $\delta n$
are   in the orbit of $x$ up to $f_\vp(\delta) n$. See Remark \ref{ucon} for more on this point.

\n {\bf Notation:} Let $E$ be a finite dimensional operator space. Fix $C>0$.
We denote by $ k_E(N,C) $ the smallest integer $k$ such that   there is a subspace $F$ of $M_N \oplus\cdots \oplus M_N$ (with $M_N$ repeated $k$-times)
        such that $d_N(E,F)\le C$.\\        
        Note that for any   $E\subset M_n$ we have $ k_E(N,1)=1 $ for any $N\ge n$.
        
        The next statement is our main result in this \S.
        It gives a lower bound
        for  $k_{E}(N,C_1)$ when $E=OH_n$. We will show later
        (see Lemma \ref{f}) that a similar upper bound holds
        for {\it all} $n$-dimensional operator spaces. Thus for $E=OH_n$
        (and also for $E=\ell_1^n$ or $E=R_n+C_n$, see Remark \ref{more})
       the growth of $N\mapsto k_{E}(N,C_1)$ is essentially {\it extremal}.
        
                \begin{thm} \label{os}There are numbers $C_1 >1$ , $b>0$ , $n_0>1$            such that for any $n\ge n_0$ and $N\ge 1 $,
         we have
        $$k_{OH_n}(N,C_1)\ge \exp{ b nN^2}.$$
        \end{thm}
        
        We start by recalling the classical argument dealing with  the Banach space case, i.e. the case $N=1$. Let 
        $E$ be an $n$-dimensional Banach space. 
        Assume that,  for some $C>1$, $E$ embeds $C$-isomorphically into $\ell_\infty^k$. For convenience we write
        $C=(1-\delta)^{-1}$ for some $\delta>0$. Our embedding assumption means that there is a set ${\cl T}$ in the unit ball of $E^*$ such
        that for any $x\in E$ we have  \begin{equation}\label{eq33}
        (1-\delta) \|x\| \le \sup_{t\in {\cl T}} |t(x)|  \le \|x\|.\end{equation}
        Then for any $x$ in the unit ball of $E$, there is $t_x\in {\cl T}$  and $\omega_x\in \CC$ with $|\omega_x|=1$ such that
        $1-\delta \le  \Re(\omega_x t_x(x) ) $.\\
        Now assume $E=\ell_2^n$.   Then   identifying       $E$ and $E^*$ as usual, 
        we see that $1-\delta \le \Re(\omega_x t_x(x) )$ implies $\|x-\omega_x t_x\|^2 \le 2 \delta$. 
         In the case of real Banach spaces, $\omega_x=\pm 1$ and we conclude quickly, but let us continue for the sake of analogy with the
        case $N>1$.
        We just proved that
        the set $\{\omega  t\mid \omega \in \T, t\in {\cl T}\}$ is a $\sqrt{2\delta}$-net in the unit ball of $E=\ell_2^n$. Fix $\vp>0$.
        Let $N(\vp)\approx 2\pi/\vp$  be  such that there is an $\vp$-net in $\T$.
        It follows that there is a $(\sqrt{2\delta}+\vp)$-net $\cl N$  in the unit ball of $E=\ell_2^n$ with $| \cl N|\le N(\vp) |{\cl T}|.$
        But  by a well known volume estimate (see e.g. \cite[p. 49-50]{P-v} ), any $\delta'$-net in the unit ball of $E=\ell_2^n$
        must have cardinality at least $(1/\delta')^n$.
        Thus we conclude $(\sqrt{2\delta}+\vp)^{-n}\le N(\vp) |{\cl T}|$.  This yields
        $$(2\pi)^{-1}\vp (\sqrt{2\delta}+\vp)^{-n} \le |{\cl T}|.$$
       For any $\delta<1/2$, we may choose $\vp>0$ so that  $\sqrt{2\delta}+\vp<1$,   thus we find that there is a number $b>0$
        for which we obtain $ |{\cl T}| \ge \exp bn$, and hence        $k_{OH_n}(1,(1-\delta)^{-1})\ge \exp{ b n}.$
        \begin{rem}\label{ucon} The preceding argument still works when $E$ is uniformly convex with modulus $\vp\mapsto \delta(\vp)$.
        This means that if $x_1,x_2$  in the unit ball $B_E$  satisfy $\|x_1-x_2\|\ge \vp$ then $\|(x_1+x_2)/2\|\le 1-\delta(\vp)$.
        Indeed, the only property we used is that for any $\vp>0$ there is $r>0$ such that
        $x_1,x_2\in B_E$ and $\xi_1,\xi_2\in B_{E^*}$ satisfy $$\Re( \xi_1(x_1))>1-r\quad  \Re( \xi_2(x_2))>1-r \quad \text{and}\quad \|\xi_1-\xi_2\|<r,$$         
        then we must have $\|x_1-x_2\|<\vp$. To check this note that
        $$\|(x_1+x_2)/2\|\ge |\xi_1   (x_1+x_2)/2 |\ge |\xi_1   (x_1 )/2+\xi_2(x_2 )/2|- \|\xi_1-\xi_2\|/2>1-r -r/2$$
        thus if $r=\delta(\vp)/2$ then we have
        $\|(x_1+x_2)/2\|>1-\delta(\vp)$ and hence we must have  $\|x_1-x_2\|<\vp$.\\
        Recall that a Banach space $E$ is uniformly convex iff its dual $E^*$ is uniformly smooth (see \cite{BL}). Thus since $E=OH_n$ is self dual, Lemma \ref{lem7} can be interpreted as the $M_N$-analogue of the uniform smoothness of $E^*$.
        \end{rem}
        
        A completely different proof, with no restriction on $\delta$ or equivalently on the constant $C$ can be given 
        by a well known argument using real or complex Gaussian random variables.
         We restrict to the real case for simplicity. 
        Let $\gamma_n $ be the canonical Gaussian measure on $\RR^n$.
        Assume  \eqref{eq33}.  Let $q= \int \exp (x^2/4) \gamma_1 (dx)<\infty$.  Note that since ${\cl T}$ is included  in the unit ball 
        we have
        $$\int  \exp (\sup_{t\in {\cl T}} t(x)^2/4) \gamma_n (dx)\le \sum\nolimits_{t\in {\cl T}}  \int  \exp (t(x)^2/4) \gamma_n (dx)\le q |{\cl T}| .$$
        But by \eqref{eq33}, if we reset  $C=(1-\delta)^{-1}$, we find $C^{-1}\|x\|\le \sup_{t\in {\cl T}}| t(x)|$ and hence
        $$(\int \exp(C^{-2}   |x|^2/4 )\gamma_1 (dx))^n  \le \int \exp(C^{-2} \sum |x_j|^2 /4) \gamma_n (dx)\le \int  \exp (\sup_{t\in {\cl T}} t(x)^2/4) \gamma_n (dx)\le q |{\cl T}| .$$
        Thus  if we define $b=b_C>0$ by  $\int \exp(C^{-2}   |x|^2/4 )\gamma_1 (dx)=\exp b$, we find
        $|{\cl T}| \ge q^{-1} \exp n b$ and we conclude
        $$   k_{OH_n}(1,C)\ge q^{-1}  \exp{ b_C n}.  $$
        See \cite{Psub} for random matrix versions of this argument.
          \begin{proof}[Proof of Theorem \ref{os}] The proof follows the strategy
          of the first proof outlined above for $N=1$, but using Theorem \ref{goal}    instead
          of the lower bound on the metric entropy of the unit ball of $\ell_2^n$. 
          Consider an $n$-dimensional operator space  $E$.
          Let  $k=k_{E}(N,C).$ Let again $C=(1-\delta)^{-1}$. Then
           there is a   set ${\cl T}$ with $|{\cl T}|=k$ and completely contractive mappings
           $\phi_t:\ E \to M_N$ such that
            \begin{equation}\label{eq4}\forall x\in M_N(E)\quad (1-\delta) \|x\|_{M_N(E)}\le \sup\nolimits_{t\in {\cl T}} \|(\phi_t)_N (x)\|_{M_N(M_N)}. \end{equation}
           Let $e_j$ be a basis for $E$ 
           so that each $x$ can be developed as $x=\sum x_j \otimes e_j\in M_N\otimes E$. Let 
           $y(t)\in M_N(E^*)$ be the element associated to    $\phi_t:\ E \to M_N$.
           Let $e_j^+\in E^*$ be the basis of $E^*$ that is biorthogonal to $(e_j)$.
           Then $y(t)$ (or equivalently $\phi_t$) can be written
           as $y(t)=\sum y_j(t)\otimes e_j^+\in M_N\otimes E^*$, and
              \eqref{eq4} can be rewritten as:
              \begin{equation}\label{eq4bis}\forall x \in M_N(E)\quad (1-\delta) \|x\|_{M_N(E)}\le \sup\nolimits_{t\in {\cl T}} \|\sum x_j \otimes y_j(t)\|_{M_N(M_N)}.  \end{equation}
             Moreover each $y(t)$ is in the unit ball of $M_N(E^*)=CB(E,M_N)$.
             We now assume $E=OH_n$.  Let  us denote by $T(\vp,\d)\subset U(N)^n$                the set  appearing in Theorem \ref{goal}. Fix $0<\vp<1$. Let us also fix a number $0<\d_0<1$.
             By Lemma \ref{lem7} we can choose $0<\d<1$ small enough so that
              \begin{equation}\label{eqrev5} 2f_{\vp}(2\delta)<\sqrt{2\delta_0 }
              . \end{equation}
             We then set $T_0=T(\vp,\d_0)$. Thus we
             have
             $| T_0 |\ge \exp {\beta_0 nN^2}$ for some
            $\beta_0>0$ and the elements of $T_0$ are
            $\d_0$-separated.
             By \eqref{eq4bis}  
              $$\forall x=(x_j) \in T_0\quad (1-\delta) n^{1/2}\le \sup\nolimits_{t\in {\cl T}} \|\sum x_j \otimes y_j(t)\|_{M_N(M_N)}. $$
              Let $ (v_j(t))=(n^{1/2} \ovl{ y_j(t)})$   so that we have
                  $$\forall x=(x_j) \in T_0\quad (1-\delta) n \le \sup\nolimits_{t\in {\cl T}} \|\sum x_j \otimes \ovl{ v_j(t)}\|_{M_N(M_N)}. $$
              For any $x\in  T_0$ there is
             a point $t_x\in \cl T$ such that 
             $$  (1-\delta) n \le   \|\sum x_j \otimes \ovl{ v_j(t_x)}\|.  $$
           Let $v_x=( v_j(t_x))$.  By  Lemma \ref{lem7}, the last inequality implies $d'(x,v_x)< f_\vp(\delta') \sqrt n$ for any $\delta'>\delta$. Moreover by the converse (much easier) part of Lemma \ref{lem77}, we know that $d'(x,y)\ge \sqrt{2\delta_0 n}$
            for any $x\not= y\in T_0$, since $x,y$ are $\delta_0$-separated. We claim that after suitably adjusting the parameters $\delta,\vp$
            we have $|T_0|\le |\cl T|$. Indeed, assume that $|T_0|> |\cl T|$, then
            there must exist $x\not=y\in T_0$ such that $v_x=v_y$. We have then for any $\delta'>\delta$
            $$  \sqrt{2\delta_0 n} \le d'(x,y)\le  d'(x,v_x)+d'( v_x,y)=d'(x,v_x)+d'( v_y,y)\le 2f_\vp(\delta') \sqrt n$$
            and hence    $\sqrt{2\delta_0 }\le2f_\vp(2\delta)$, which is impossible by \eqref{eqrev5}. This
            proves our claim
  that $|T_0 |\le |\cl T|$,  and hence $  |\cl T|\ge \exp \beta_0 nN^2$.
  Let $C_1=(1-\delta)^{-1}$. 
           Thus, with $\d$ determined by \eqref{eqrev5}, we have proved $ k_{OH_n}(N,C)\ge \exp \beta_0 nN^2$.
          \end{proof}

          \begin{rem}\label{more} Let $E$ be any $n$-dimensional operator space with a basis $(e_j)$.
 Assume that there is a scaling factor ${\lambda}>0$ (that does not play any role in the estimate) such that for any $u=(u_j)\in U(N)^n$ we have
 ${\lambda}\|\sum u_j\otimes e_j\|_{M_N(E)}=\sqrt n$ and also that 
 ${\lambda}\|\sum a_j\otimes e_j\|_{M_N(E)}\le \|\sum a_j\otimes \bar a_j\|^{1/2}$ for all $a=(a_j)\in M_N^n$.
Then,   arguing as in Remark \ref{rem}, we find $ k_{E}(N,C_1)\ge \exp \beta nN^2$.
This shows that this estimate is valid for $R_n+C_n$  (take ${\lambda}=1$) and for $\ell_1^n$ equipped with its
maximal operator space structure (take ${\lambda}=n^{-1/2}$). 
\end{rem}
          
          We now turn to the reverse  inequality     to that in    Theorem \ref{os}.             This general estimate is
          easy to check by a rather routine argument.
          
            \begin{lem}\label{f} Let $E$ be an $n$-dimensional operator space, then for any $0<\delta <1$ 
   we have
   $$  k_E(N,(1-\delta)^{-1})\le (1+2\delta^{-1})^{2nN^2}   .$$
   Therefore, for any operator space $X$, any finite dimensional subspace
   $E\subset X$   we have
   $$\forall C>1 \quad    \limsup_{N\to \infty} \frac{\log k_E(N,C)} {N^2}<\infty.$$
   \end{lem}
   \begin{proof} Let $x\in M_N(E)$ and let $\hat x:\ E^* \to M_N$ denote the associated linear mapping.
   Recall $ \|x\|=\|\hat x\|_{cb}$.   
   By Lemma \ref{rs} 
   $\|\hat x\|_{cb}=\sup\{ \|(\hat x)_N (y)\|_{M_N(M_N)}\mid y \in B_N\}$ where we denote here by $B_N$ the unit ball
   of $M_N(E^*)$ viewed as a real space. Since the latter ball is $2nN^2$-dimensional, it contains a $\delta$-net
   $\{y_i\mid i\le m\}$ with cardinality $m\le (1+2\delta^{-1})^{2nN^2} $ (see e.g. \cite[p. 49-50]{P-v}). By an elementary estimate,
   we have then 
   (for any $x\in M_N(E)$) 
     \begin{equation}\label{eq333}\sup_{i\le m}  \|(\hat x)_N (y_i)\|\le \|\hat x\|_{cb}=\|x\|\le (1-\delta)^{-1} \sup_{i\le m}  \|(\hat x)_N (y_i)\|.\end{equation}
   Let $u:\ E \to \oplus_{i\le m} M_N$  be the mapping defined by (here again $\hat{y_i}:\ E\to M_N$ is associated to $y_i$)
   $$u(e)= \oplus_{i\le m}  \hat{y_i}(e)$$ for any $e\in E$.
   Let $F\subset \oplus_{i\le m} M_N$ be the range of $u$.
   Then \eqref{eq333} says that $\|u_N\|\le 1$
   and  $\|u^{-1}_N\|\le 1+\delta$, and hence $d_N(E,F)\le (1-\delta)^{-1}$. Thus $k_E(N,(1-\delta)^{-1})\le m$.
      \end{proof}

          \begin{dfn}\label{def} An operator space $X$ will be  called matricially
 $C$-subGaussian     if
  $$\limsup_{N\to \infty} \frac{\log k_E(N,C)} {N^2}=0.$$
for any finite dimensional subspace $E\subset X$. We say that $X$ is    matricially
subGaussian
if it is matricially $C$-subGaussian for some $C\ge1$. (See Remark \ref{exp} for 
the reason behind ``matricially"). \\
{\bf Note:} If $X$ itself is finite dimensional, it suffices to consider $E=X$.\\
 We will denote by $C_g(X)$ the smallest $C$ such that 
$X$ is matricially $C$-subGaussian.
\end{dfn}

The preceding result (resp. Remark \ref{more}) shows that when $C<C_1$, then $OH$ 
(resp.  $\ell_1$ or $R+C$) is not matricially $C$-subGaussian. In sharp contrast, any $C$-exact operator space
(we recall the definition below) $E$ is clearly matricially $C$-subGaussian since,
for any $c>C$, it satisfies   $  k_E(N,c)=1$ for all $N$ large enough. We do not know whether conversely
the latter property implies
that $E$ is $C$-exact (but we doubt it). 

  \begin{rem} Given an operator space $X$, it is natural to introduce the following parameter:
    $$k_X(N,C;d)=\sup\{ k_E(N,C) \mid E\subset X,\ \dim(E)=d\}.$$
       We will say that $X$ is     uniformly matricially subGaussian
     if there is $C$ such that 
    $$\forall d\ge1\quad   \limsup_{N\to \infty} \frac{\log k_X(N,C;d)} {N}=0.$$
It is   easy to check that if $X$ is     uniformly  exact (resp.  uniformly subexponential, rresp. 
 uniformly matricially  subGaussian) then all ultrapowers of $X$ are exact
  (resp.   subexponential, rresp. matricially
  subGaussian).
  Note however (I am indebted to Yanqi Qiu for this remark) that the converse is unclear.\\
  For example, $R$ or $C$ (or $R\oplus C$),   any commutative
  $C^*$ algebra $A$, or  any space of the form $A \otimes_{\min} M_N$  
  is       uniformly  exact. It would be interesting to
  characterize   uniformly  exact operator spaces. \end{rem}

We now turn to a different application of quantum expanders to operator spaces, that 
requires a refinement of our main result.

For any $n\times n$   matrix   $w$   and any $v\in M_N^n$,
we denote by $w.v \in M_N^n$ the $n$-tuple defined by
$$(w.v)_i =\sum\nolimits _j w_{ij} v_j.$$
Note that if $w$ is unitary, i.e. $w\in U(n)$ then
 \begin{equation}\label{eq21}\sum\nolimits_i (w.v)_i \otimes \overline{ (w.v)_i } =\sum\nolimits_j v_j \otimes \bar v_j.\end{equation}
Also note that, if $w\in U(n)$, for any $v,v'\in M_N^n$ we have 
 \begin{equation}\label{eq22}d(w.v,w.v')\le d(v,v').\end{equation}
Moreover, it is easy to check (e.g. using \eqref{25}) that
for all   $w\in M_n$ with operator norm $\|w\|$ and for all $v\in M_N^n$
we have
 \begin{equation}\label{eq23}\|\sum (w.v)_i \otimes \overline{ (w.v)_i } \|\le \|w\|^2 \|\sum\nolimits_j v_j \otimes \bar v_j\|,\end{equation}
and hence by \eqref{eqcsh} for any $u,v\in U(N)^n$
 \begin{equation}\label{eq24}\|\sum u_i \otimes \overline{ (w.v)_i } \|\le \|w\| n.\end{equation}
Also
$$d(w.v,w.v')\le \|w\| d(v,v').$$
We will say that $u,v\in U(N)^n$ are strongly $\delta$-separated if 
$v$ and $w.u$ are $\delta$-separated for any $w\in U(n)$. Equivalently, for any pair $w,w'\in U(n)$
the pair $(w.u,w'.v)$ is $\delta$-separated. \\
Explicitly, this can be written like this:
 \begin{equation}\label{26}  \forall w \in U(n)\quad
 \| \sum\nolimits_{ij}  w_{ij} u_j \otimes \bar v_i \|\le   n(1-\delta).
 \end{equation}
We will use again (see Lemma \ref{lem6}) the following elementary fact : There is a positive constant $D$ such that
for each $0<\xi<1$ and each $n$ there is an $\xi$-net  $\cl N_\xi\subset U(n)$ 
with respect to the operator norm, of cardinality
$$|\cl N_\xi|\le (D/\xi)^{2n^2}.$$

We will need the following refinement of Theorem \ref{goal}.
\begin{lem} For each $0<\d<1$ there is a  constant $\beta'_{\d}>0$   such that for any  $0<\vp<1$   and
for all  $n\ge n_0$ and all $N$ such that $N^2/n\ge N_0$ (with $n_0$ depending
on   $\vp$ and $\d$, and $ N_0$  depending on $\d$),  there is a strongly $\delta$-separated subset $T_1\subset S_\vp$  such that
$|T_1|\ge \exp{\beta'_{\d} nN^2}.$ More generally, for each $\alpha>0$, there are  $\beta'_{\d,\alpha}>0$ and $n_0=n_0(\vp,\d,\alpha)$ such that, if  $n\ge n_0$ and $N^2/n\ge N_0$, any subset 
$A_N\subset U(N)^n$ with $\P(A_N)>\alpha$ contains a strongly $\delta$-separated subset of $S_\vp$
with  cardinal $\ge \exp{\beta'_{\d,\alpha} nN^2}$.
 \end{lem} 
 \begin{proof} Fix $0<\d<1$ and let $\xi=(1-\d)/2$
so that $\d_1=\d+\xi=(1+\d)/2$. Note that $0<\d<\d_1<1$.
We define $\vp_0$ so that $2\vp_0^{1/2}=(1-\d_1)/2$ and $\vp'$ so that
$\vp'^{1/5} (2^{-4/5}+2^{6/5})=(1-\d_1)/2$.
Note that $0<\vp_0,\vp'<1$ and 
$$1-\d_1=\vp'^{1/5} (2^{-4/5}+2^{6/5}) +2\vp_0^{1/2}.$$
Now assume $0<\vp\le \vp_0$.  By Lemma \ref{lem77} we know that for any $u\in S_\vp$ and any
$v\in U(N)^n$ such that $\|\sum v_j \otimes \bar v_j\|\le n$ we have 
\begin{equation}\label{eq007}
\|\sum u_j \otimes \bar v_j\| >n(1-\d_1) \Rightarrow d'(u,v)<\sqrt{2n(1-\vp')}.
\end{equation}
By Lemmas \ref{lem5} and  \ref{lem6} and using \eqref{a44} 
as in the proof of Theorem \ref{goal} we know that for $n\ge n_0(\vp,\vp')$
$$N(S_\vp,d', \sqrt{2n(1-\vp')}) \ge \exp{ b'nN^2}$$
for some $b'$ depending only on $\d$ (more precisely
we set again $r=\vp'^2$,   $b=K\vp'^2/2$ and  $b'=b/2$).

Let $T_1\subset S_\vp $ be a maximal subset
such that any two points in $T_1$ are strongly  $\delta$-separated. By maximality of $T_1$
for any $u\in S_\vp$ there is $x\in T_1$ such that $u,x$ are not strongly $\delta$-separated.
This means that there is $w\in U(n)$ such that
$$\|\sum u_j \otimes \overline{(w.x)_j} \|> n(1-\delta).$$
Choose $w'\in \cl N_\xi$ such that $\|w-w'\|\le \xi$.
Then by \eqref{eq24} and the triangle inequality we have
$$\|\sum u_j \otimes \overline{(w'.x)_j} \|\ge \|\sum u_j \otimes \overline{(w.x)_j} \|-n\xi  >n(1-\delta -\xi)=1-\d_1.$$
By \eqref{eq007} it follows that $d'(u, w'.x)<\sqrt{2n(1-\vp')}$.
In other words, we find that the set $T_2=\{w'. x\mid w'\in\cl N_\xi,  \ x\in T_1\}$
is a $\sqrt{2n(1-\vp')}$-net for $S_\vp$, and hence
$$\exp{ b'nN^2}\le N(S_\vp,d', \sqrt{2n(1-\vp')}) \le |T_2| \le  |\cl N_\xi| |T_1|\le (D/\xi)^{2n^2} | T_1|$$
This yields 
  $$|T_1|\ge (2D/(1-\delta))^{-2n^2} \exp b' n N^2.$$
  Assuming $\vp\le \vp_0$, this completes the proof,    since for $N^2/n\ge N_0(\d)$  the first factor
  can be absorbed, say, by choosing $\beta'_\d=b'/2$. The case  $\vp_0< \vp<1$ follows a fortiori since
  $S_{\vp_0} \subset S_\vp$.\\
   The   last assertion follows  (for  suitably adjusted values of $\beta'_\d$ and $n_0$) as in Remark \ref{mg}.
   Indeed,   choosing $n_0$ large enough (depending on $\alpha$) we can make sure that $\P(S_\vp)>1-\alpha/2$ so that
   $\P(A_N \cap S_\vp)>\alpha/2$. We can then run the preceding proof using the set $A_N \cap S_\vp$
   in place of $S_\vp$.
 \end{proof}

\begin{thm}\label{ent}  
For any $R>1$,    there are numbers $\beta_1>0$,   $n_0>1$ and a function
        $n\mapsto N_0(n)$ from $\NN $ to itself 
          such that for any $n\ge n_0$ and $N\ge N_0(n) $, there is a 
family  $\{E_t\mid t \in T_1\}$ of $n$-dimensional subspaces of $M_N$, 
with cardinality 
$|T_1|\ge \exp{\beta_1 nN^2},$ such that for
any $s\not= t\in T_1$ we have
$$d_{cb}(E_s,E_t)>R.$$
\end{thm}
   \begin{proof} Fix $0<\d<1$. We will prove this for $R=(1-\d)^{-1}$.
   We will use the set $T_1$ from the preceding Lemma
 and we let $E_t={\rm span}\{t_1,\cdots,t_n\}$.  
 We may clearly assume (say by perturbation)
 that    $\{t_1,\cdots,t_n\}$ are linearly independent for all $t\in T_1$
 so that $\dim(E_t)=n$ (but this will be automatic, see below).
 Consider $s\not= t\in T_1$. 
 Let $W\in M_n$, and let $W:\ E_s\to E_t$ denote the associated  linear map
 so that $W s_j=\sum\nolimits_i W_{i j} t_i$.\\
We claim that   we can ``make sure" that for all $N$ large enough
$${\rm tr}|W | \le n (1-\delta)^{1/2}.$$
We first clarify what we mean here by ``$N$ large enough". Let $0<\gamma_1<1$ be such that
 \begin{equation}\label{30}(1-\gamma_1)^{-1} (1-\delta) =(1-\delta)^{1/2}, \end{equation}
let
$$ \Delta_{N,n}(t)=n^2\sup\nolimits_{i\not= j}|  {\tau_N} (t_i t_j^*)|.$$ Then
we require that $N$ is large enough (depending on a fixed $n$) so that 
with respect to the uniform probability on $U(N)^n$ we have 
 \begin{equation}\label{29}
 \P\{ t\in U(N)^n\mid \Delta_{N,n}(t) <\gamma_1\}> 1/2.
 \end{equation}
 Clearly this is possible because, by the almost sure weak convergence, 
 we know that ${\tau_N} (t_i t_j^*)\to 0$
 when $N\to \infty$ for any $1\le i\not= j\le n$.\\
 Using  the last assertion in the preceding Lemma,
 we see that we may assume 
 $$\forall t\in T_1\quad  \Delta_{N,n}(t) <\gamma_1.$$
 To verify the above claim, we will use an idea from \cite{OR} (refining one in \cite{JP}). 
First
we note that for any matrix $a=[a_{ij}]$ and  for any $t\in U(N)^n$, if we assume
  ${\tau_N}(t_i t_j^*)=0$ for all $i\not= j$, then  we have by  \eqref{25}
$$|{\rm tr}(a)|\le \|\sum a_{ij} t_i \otimes \bar t_j\|.$$
More generally, 
with the notation from \eqref{25},
 we have $\langle \sum a_{ij} t_i \otimes \bar { t_j}(I),I\rangle=
\sum a_{ii}  +\sum\nolimits_{i\not= j} a_{ij}   {\tau_N} (t_i t_j^*) $ and 
$|\sum\nolimits_{i\not= j} a_{ij}   {\tau_N} (t_i t_j^*)| \le \Delta_{N,n}  \sup\nolimits_{i\not= j} |a_{ij}|.$
Therefore, without this assumption, we still have 
 \begin{equation}\label{28}|{\rm tr}(a)|\le  \|\sum a_{ij} t_i \otimes \bar t_j\|  +\gamma_1 \|a\|_1.
 \end{equation}
 By \eqref{26} and an extreme point argument (since the unitaries are the extreme points of the unit ball
of $M_n$) we have
for any $s\not= t\in T_1$ and any $w\in M_n$
 \begin{equation}\label{27} \| \sum \bar w_{ij} s_i \otimes \bar t_j \|\le \|w\|  n(1-\delta). \end{equation}
Now  we can write for any  $W:\ E_s\to E_t$ by \eqref{27}
 \def\o{\overline}
 $$\|\sum W s_j \otimes \o{  (w.t)_j} \|\le \|W\|_{cb} \|\sum   s_j \otimes \o{  (w.t)_j} \|\le 
  \|W\|_{cb} \|w\|n(1-\delta)$$
 Therefore
 $$\|\sum\nolimits_{ijk} W_{ij} \o{w_{jk} } t_i \otimes \o{ t_k}   \|\le 
  \|W\|_{cb} \|w\|n(1-\delta)$$
 hence (replacing $w$ by its transpose) by \eqref{28}  we have 
 $$ |{\rm tr}  (Ww^*)|   \le   \|W\|_{cb} \|w\|n(1-\delta) +\gamma_1 \|Ww^*\|_1,
 $$
 and hence taking the sup over all $w\in U(n)$
 $$\|W \|_1={\rm tr}|W|   \le    \|W\|_{cb} n(1-\delta) +\gamma_1 \|W \|_1.
 $$
 Thus, we conclude by \eqref{30}
  \begin{equation}\label{31}  {\rm tr}|W|   \le  \|W\|_{cb} n (1-\gamma_1)^{-1}   (1-\delta)=   n \|W\|_{cb}(1-\delta)^{1/2}.
\end{equation}
 Applying \eqref{31} with $W^{-1}$ in place of $W$ we find 
 $$  {\rm tr}|W^{-1}|   \le     n \|W^{-1}\|_{cb}(1-\delta)^{1/2},$$
 and hence
 $$ {\rm tr}|W| {\rm tr}|W^{-1}| \le  n^2 \|W\|_{cb} \|W^{-1}\|_{cb}(1-\delta),$$
 but we will immediately justify that any invertible matrix in $M_n$ satisfies
  \begin{equation}\label{32} n^2\le  {\rm tr}|W| {\rm tr}|W^{-1}| ,\end{equation}
 so that we obtain
 $$d_{cb}(E_s,E_t)\ge (1-\delta)^{-1}=R.$$
To check \eqref{32} recall that for any pair $W_1,W_2\in M_n$
the Schatten $p$-norms $\|.\|_p$
 satisfy whenever $0<p,q,r$ and $1/r=1/p+1/q$
$$\|W_1W_2\|_r\le \|W_1 \|_p \| W_2\|_q. $$
Moreover  $\|I\|_r=n^{1/r}$. 
Therefore, \eqref{32} follows by taking $r=1/2$ and $p=q=1$.
 \end{proof}
   \section{Random matrices and subexponential operator spaces}\label{s3}

In a forthcoming  sequel to this paper \cite{Psub}, we   introduce and  study  a generalization of the notion of  exact operator space
that we call subexponential. We briefly outline this here.

Our goal is to study a generalization of the notion of exact operator space for which the version of Grothendieck's theorem obtained in \cite {PS} is still valid.

\n {\bf Notation:}  Let $E$ be a finite dimensional operator space. Fix $C>0$.
We denote by $K_E(N,C)$ the smallest integer $K$ such that
there is an operator subspace $F\subset M_K$
such that
$$d_N(E,F)\le C.$$

Note that obviously

 \begin{equation}\label{comp}K_E(N,C)\le N k_E(N,C). \end{equation}

\begin{dfn} We say that an operator space $X$ is $C$-subexponential
if $$\limsup_{N\to \infty} \frac{\log K_E(N,C)} {N}=0,$$
for any finite dimensional subspace $E\subset X$.
We say that $X$ is    subexponential
if it is $C$-subexponential for some $C\ge1$.\\
{\bf Note:} If $X$ itself is finite dimensional, it suffices to consider $E=X$.\\
 We will denote by $C(X)$ the smallest $C$ such that 
$X$ is $C$-subexponential.
\end{dfn}

Recall that an operator space $X$ is called $C$-exact 
if for any finite dimensional subspace $E\subset X$
and any $c>C$ there is a $k$ and $F\subset M_k$
such that $d_{cb}(E,F)<c$. We denote by $ex(X)$ the smallest such $C$.
We say that $X$ is exact if it is $C$-exact for some $C\ge1$.\\
We observe in \cite{Psub} that  a finite dimensional $E$ is $C$-exact iff
for any $c>C$ the sequence $N\mapsto K_E(N,c)$ is bounded.
In this light ``subexponential" seems considerably more general than ``exact".

As shown by Kirchberg, a $C^*$-algebra is exact iff it is $1$-exact.
We do not know whether the analogue of this for subexponential  (or for matricially subGaussian) $C^*$-algebras  is true.
See \cite[ch.17]{P4} or \cite{BO} for more background on exactness.

In \cite{Psub} we show that for essentially all the results proved in either \cite{JP} or \cite{PS}
we can replace    exact by subexponential in the assumptions. Moreover, we show
that there is a $1$-subexponential $C^*$-algebra that is not exact.
    
                    \begin{rem}\label{exp} In the same vein, it is natural to call an operator space $X$  $C$-subGaussian
if \\
$\limsup_{N\to \infty} {N^{-2}} {\log K_E(N,C)} =0$
for any finite dimensional subspace $E\subset X$. We do not have 
significant information about this class at this point, but 
to avoid confusion, we decided to call ``matricially subGaussian"
the spaces in Definition \ref{def}. Clearly by \eqref{comp} ``matricially subGaussian"
implies ``subGaussian" but the converse is unclear.
\end{rem}

   \n {\bf Problems:}\\
    1) Let $C>1$. Assume that a finite dimensional space $E$ satisfies $k_E(N,C)\le 1$
   for all $N$. What does that imply on $E$ ? Is $E$ exact  with a control on its exactness constant ? \\
   2) Assume   $E$   subexponential for some constant $C$. What growth does that imply for  $N \mapsto k_E(N,C)$ (here $C$ could be   a different constant) ?\\
     3) What is the order of growth (when $N\to \infty$) of $\log K_E(N,C)$ for $E=\ell_1^n$ or $E=OH_n$ ? In particular, when $C$ is close to $1$,
   is it $O(N)$ ? or to the contrary does it grow  like $N^2$ ?

         \bigskip
  
      \section    {\bf Appendix}
         \def\o{\overline}
            \def\tr{{\rm tr }}
       
         \def\o{\overline}
            \def\tr{{\rm tr }}
        In this appendix we give a quick proof of an inequality that can be substituted in \S  2 to Hastings's
        result from \cite{Ha}, quoted  above as  Lemma \ref{has}. Our inequality is less sharp
        in some respect but stronger in some other. We only prove
        that (for some numerical constant $C$)   $\P\{(u_j)\in U(N)^n\mid \|(\sum u_j \otimes \bar u_j)(1-P)\| >4C\sqrt{n}+\vp n\}\to 1$ when $N\to \infty$
        for any $\vp>0$, while Hastings proves this with 
        $2\sqrt{n-1}$ in place of $4C\sqrt{n}$ which is best possible. However the inequality below
         remains valid with more general (and even matricial) coefficients,
         and it gives a bound valid uniformly for all sizes $N$ (see \eqref{a44}). It shows that up to a universal constant
        all moments of the norm of a linear combination of  the form $$S=\sum\nolimits_j   a_j  U_j \otimes \bar U_j (1-P)$$ are dominated
        by those of  the corresponding Gaussian sum
        $$S'=\sum\nolimits_j  a_j Y_j \otimes \bar Y'_j .$$
        The advantage is that  $S'$ is now simply separately  a Gaussian random variable
        with respect to the independent Gaussian random  matrices $(Y_j)$ and $(Y'_j)$.\\
        We recall that we denote by $P$
        the orthogonal projection onto the multiples of the identity. 
        Also recall we denote by $S_2^N$ the space $M_N$ equipped with the Hilbert-Schmidt norm  (recall    $S_2^N\simeq \ell_2^N \otimes_2 \o{\ell_2^N}$).
        We will view elements
        of the form $\sum x_j \otimes \bar y_j$ with $x_j,y_j\in M_N$
        as linear operators acting on $S_2^N$  as follows
        $$T(\xi)= \sum\nolimits_j x_j\xi y_j^*,$$
        so that      \begin{equation}\label{i34}\|\sum x_j \otimes \bar y_j\|=\|T\|_{B(S_2^N)}.\end{equation}
        
        We denote by $(U_j)$   a sequence of i.i.d. random $N\times N$-matrices uniformly distributed over the unitary group $U(N)$.
        We will denote by  $(Y_j)$ a sequence of i.i.d. Gaussian random $N\times N$-matrices, more precisely each $Y_j$ is distributed like the variable $Y$
        that is such that  $\{Y(i,j) N^{1/2}\}$ is a standard family of $N^2$ independent complex Gaussian variables with mean zero and variance 1.
        In other words $Y(i,j) =(2N)^{-1/2}(g_{ij} + \sqrt{-1} g'_{ij})$ where  $g_{ij} ,g'_{ij} $ are independent Gaussian normal $N(0,1)$ random variables.
        
        We denote by $(Y'_j)$ an independent copy of $(Y_j)$.
        
        We will denote by $\|.\|_q$ the Schatten $q$-norm ($1\le q\le \infty$), i.e.
        $\|x\|_q=(\tr(|x|^q))^{1/q}$, with the usual convention that
        for $q=\infty$ this is the operator norm.
        \begin{lem} There is an absolute constant $C$
        such that for any $p\ge 1$ we have for any scalar sequence $(a_j)$ and any $1\le q\le \infty$
        $$\E\|\sum\nolimits_1^n  a_j U_j \otimes \bar U_j (1-P)\|_q^p \le C^p \E \|\sum\nolimits_1^n a_j Y_j \otimes \bar Y'_j \|_q^p,$$
        (in fact this holds for all $k$ and all matrices $a_j\in M_k$
        with $   a_j\otimes $ in place of $a_j$).
        \end{lem}
           \begin{proof} We assume that all three sequences $(U_j)$, $(Y_j)$ and $(Y'_j)$ are mutually independent.
           The proof is based on the well known fact that the sequence $(Y_j)$ has the same distribution
           as $U_j|Y_j|$, or equivalently that the two factors
           in the polar decomposition $Y_j=U_j|Y_j|$ of $Y_j$ are mutually independent. Let $\cl E$  denote the conditional expectation operator
           with respect to the $\sigma$-algebra generated by $(U_j)$. Then we have
           $U_j \E|Y_j|= {\cl E}( U_j|Y_j|)={\cl E}( Y_j)$, and moreover
           $$(U_j\otimes \bar U_j)  \E(|Y_j| \otimes \o{|Y_j|})={\cl E}( U_j|Y_j| \otimes \o{U_j|Y_j|})={\cl E}(Y_j \otimes \o{Y_j}) .$$
           Let
           $$T=\E(|Y_j| \otimes \o{|Y_j|})=\E(|Y| \otimes \o{|Y|}).$$
           Then we have
           $$ \sum a_j (U_j\otimes \bar U_j) T (I-P) ={\cl E}((  \sum a_j Y_j \otimes \o{Y_j})(I-P)).$$
           Note that by rotational invariance of the Gaussian measure
           we have
           $(U \otimes \bar U)T(U^* \otimes \bar U^*)=  T$. Indeed   since
           $UYU^*$ and $Y$ have the same distribution
           it follows that
           also $UYU^* \otimes \o{UYU^*} $ and $Y\otimes \bar Y$  have the same distribution,
           and hence so do their modulus.\\
           Viewing $T$ as a linear map on $S_2^N=\ell_2^N \otimes \o{\ell_2^N} $,
           this yields
           $$\forall U\in U(N)\quad T (U\xi U^*)= UT (\xi )U^*.$$
          Representation theory 
           shows that $T$ must be simply    a linear combination of $P$ and $I-P$.
           Indeed, the unitary representation $U\mapsto U\otimes \bar U$ 
           on $U(N)$ decomposes
           into exactly two distinct irreducibles, by restricting either
           to the subspace  $\CC I$ or its orthogonal. Thus, by Schur's Lemma we know a priori
           that there are two scalars $\chi'_N ,\chi_N$ such that $T=\chi'_N P+ \chi_N (I-P)$.
           We may also observe $\E( |Y|^2)=I$ so  that $T(I)=I$ and hence $\chi'_N=1$, therefore
           $$T=P+\chi_N (I-P).$$
           Moreover, since $T(I)=I$ and $T$ is self-adjoint, $T$ commutes with $P$ and  hence $T(I-P)=(I-P)T$, 
            so that we have
           \begin{equation}\label{37}  \sum\nolimits_1^n  a_j (U_j \otimes \bar U_j )(1-P) T = {\cl E} \sum\nolimits_1^n a_j (Y_j \otimes \bar Y_j )(I-P).\end{equation}
           We claim that $T$ is invertible and that
           there is an absolute constant $C_0$ so that
           $$\|T^{-1}\| ={\chi_N}^{-1}\le C_0.$$
           From this and \eqref{37}  follows immediately that for any $p\ge 1$
          \begin{equation}\label{36} \E\|\sum\nolimits_1^n  a_j (U_j \otimes \bar U_j) (1-P)\|_q^p \le C_0^p  \E\| \sum\nolimits_1^n a_j (Y_j \otimes \bar Y_j) (1-P)  \|_q^p.\end{equation}
        
          To check the claim it suffices to compute $\chi_N$. For $i\not=j$ we have a priori
            $T (e_{ij})= e_{ij} \langle T(e_{ij}), e_{ij}\rangle$ but (since $\tr (e_{ij}  )=0$) we know $T (e_{ij})=\chi_N e_{ij}$.
            Therefore for any $i\not=j$ we have $\chi_N = \langle T(e_{ij}), e_{ij}\rangle$,
                       and the latter we can compute:
                     $$\langle T(e_{ij}), e_{ij}\rangle =\E\tr (|Y|e_{ij} |Y|^*e^*_{ij})= \E( |Y|_{ii} |Y|_{jj}).$$
         Therefore, 
         $$N(N-1)  \chi_N= \sum\nolimits_{i\not=j}  \E( |Y|_{ii} |Y|_{jj})=\sum\nolimits_{i,j}  \E( |Y|_{ii} |Y|_{jj})-\sum\nolimits_j  \E(  |Y|^2_{jj})=\E(\tr|Y |)^2 -N \E(  |Y|^2_{11}).$$
         Note that $\E(  |Y|^2_{11})= \E  \langle |Y| e_1,e_1\rangle ^2 \le   \E \langle |Y|^2 e_1,e_1\rangle =
         \E\|Y (e_1)\|^2_2=1$, and hence
         $$N(N-1)  \chi_N= \sum\nolimits_{i\not=j}  \E( |Y|_{ii} |Y|_{jj})\ge \E(\tr|Y |)^2 -N  .$$
         Now it is well known that $E|Y| =b_N I$ where $b_N$ is determined by
         $b_N=N^{-1}\E \tr |Y|= N^{-1}\|Y\|_1$ and $\inf_N b_N >0$ (see e.g. \cite[p. 80]{MP}). Actually, by a well known limit theorem
         originating in Wigner's work (see \cite{VDN}),
         when $N \to \infty$, $N^{-1}\|Y\|_1$ tends almost surely to the $L_1$-norm denoted by $\|c\|_1$
         of a circular random variable $c$ normalized in $L_2$.
         Therefore, $ N^{-2}\E(\tr|Y |)^2 $ tends to $\|c\|_1$.  
         We have
         $$  \chi_N=(N(N-1))^{-1} \sum\nolimits_{i\not=j}  \E( |Y|_{ii} |Y|_{jj})\ge (N(N-1))^{-1} \E(\tr|Y |)^2 -(N-1)^{-1}  ,$$
         and this implies
           $$\liminf_{N\to \infty}   \chi_N \ge (\|c\|_1)^2,$$
and actually   $\chi_N \to (\|c\|_1)^2.$  In any case, we have
         $$\inf\nolimits_N \chi_N >0,$$
         proving our claim with $C_0=(\inf\nolimits_N \chi_N )^{-1}$.

                    We will now deduce from \eqref{36} the desired estimate by a classical decoupling
                    argument for multilinear expressions in Gaussian variables.\\
                    We first observe $\E ((Y \otimes \bar Y) (I-P))=0$. Indeed, 
                    by orthogonality, a simple calculation shows that  $\E (Y \otimes \bar Y)=\sum _{ij} \E (Y_{ij}\o{Y_{ij}})e_{ij}\otimes \o{e_{ij}}= \sum _{ij} N^{-1}e_{ij}\otimes \o{e_{ij}}=P $,
                    and hence $\E ((Y \otimes \bar Y) (I-P))=0$.
                    
                    We will use 
                    $$(Y_j, Y'_j ) {\buildrel {dist} \over {=}}( (Y_j+Y'_j )/\sqrt 2,  (Y_j-Y'_j )/\sqrt 2)$$
                    and
                    if $\E_Y$ denotes the conditional expectation with respect to
                    $Y$ we have (recall $\E (Y_j \otimes \bar Y_j )(I-P)=0$)
                    $$\sum\nolimits_1^n a_j Y_j \otimes \bar Y_j (I-P)= \E_Y (\sum\nolimits_1^n a_j Y_j \otimes \bar Y_j (I-P)-
                    \sum\nolimits_1^n a_j Y'_j \otimes \bar Y'_j(I-P)).$$
                  Therefore    
                   $$ \E\|\sum\nolimits_1^n  a_j Y_j \otimes \bar Y_j (1-P)\|_q^p \le
                   \E\|\sum\nolimits_1^n  a_j Y_j \otimes \bar Y_j (1-P)- \sum\nolimits_1^n a_j Y'_j \otimes \bar Y'_j(I-P))\|_q^p   $$
                    
                   $$=  \E\|\sum\nolimits_1^n  a_j (Y_j+Y'_j )/\sqrt 2 \otimes \o{ (Y_j+Y'_j )/\sqrt 2} (1-P)- \sum\nolimits_1^n a_j (Y_j-Y'_j )/\sqrt 2 \otimes \o{ (Y_j-Y'_j )/\sqrt 2}(I-P))\|_q^p $$                   
     $$= \E\|\sum\nolimits_1^n  a_j  (Y_j    \otimes \o{Y'_j }+Y'_j    \otimes \o{Y_j }  ) (1-P) \|_q^p
           $$
           and hence by the triangle inequality
           $$\le 2^p \E\|\sum\nolimits_1^n  a_j  (Y_j    \otimes \o{Y'_j }   ) (1-P) \|_q^p.$$
Thus we conclude a fortiori
$$\E\|\sum\nolimits_1^n  a_j U_j \otimes \bar U_j (1-P)\|_q^p \le (2C_0)^p \E\|\sum\nolimits_1^n  a_j  (Y_j    \otimes \o{Y'_j }   )  \|_q^p,$$
so that we can take $C=2C_0$.
 \end{proof}
    \begin{thm} Let $C$ be as in the preceding Lemma. Let 
$$\hat S^{(N)}= \sum\nolimits_1^n  a_j U_j \otimes \bar U_j (1-P).$$
Then 
 \begin{equation}\label{a4}\limsup_{N\to \infty}\E \|\hat S^{(N)} \| \le   4C ( \sum |a_j|^2)^{1/2}. \end{equation}
Moreover we have 
almost surely
 \begin{equation}\label{a5} \limsup_{N\to \infty} \|\hat S^{(N)} \| \le   4C ( \sum |a_j|^2)^{1/2}. \end{equation}
  In addition, there is a constant $C'>0$   such that for any scalars $(a_j)$
  \begin{equation}\label{a44}\forall N\ge 1\quad \E \|\hat S^{(N)} \| \le   C'( \sum |a_j|^2)^{1/2}. \end{equation}
     \end{thm}
   \begin{proof} A very direct argument is indicated in Remark \ref{sim2} below, but we prefer
   to base the proof  on \cite{HT2} in the style of    \cite{Psub} in order to make clear that it remains valid
   with matrix coefficients.
By    \cite[(1.1)]{Psub} applied twice (for $k=1$) (see also   Remark 1.5 in \cite{Psub})
one finds for 
any even integer $p$
 \begin{equation}\label{a1}
E\tr |\sum\nolimits_1^n  a_j  (Y_j    \otimes \o{Y'_j }   ) |^p  \le (\E \tr|Y|^p)^2 ( \sum |a_j|^2)^{p/2}
 \end{equation}
Therefore by the preceding Lemma
 $$E\tr |\hat S^{(N)} |^p  \le C^p (\E \tr|Y|^p)^2 ( \sum |a_j|^2)^{p/2}  ,$$
and hence a fortiori
$$E\|\hat S^{(N)} \|^p\le  N^2 C^p (\E \|Y\|^p)^2 ( \sum |a_j|^2)^{p/2} .$$
We then complete the proof, as in \cite{Psub}, using only the concentration
of the variable $ \|Y\|$. We have an absolute constant $\beta'$  and   $\vp(N)>0$ 
tending to zero when $N\to \infty$, such that
$$(\E \|Y\|^p)^{1/p}  \le 2+\vp(N) + \beta' \sqrt{p/N},  $$
and hence
$$(E\|\hat S^{(N)} \|^p) ^{1/p} \le N^{2/p}C (2+\vp(N) + \beta' \sqrt{p/N})^2 ( \sum |a_j|^2)^{1/2} .
$$
Fix $0<\vp<1$. If we  choose $p$ minimal  even integer so that $N^{2/p} \le \exp \vp $, i.e.
if we set  $p=2   ([\vp^{-1} \log N]+1)  $
(note that $p  >2\vp^{-1} \log N$ and also $p\ge 2$)
  we obtain
  $$E\|\hat S^{(N)} \|  \le (E\|\hat S^{(N)} \|^p) ^{1/p} \le 4 e^\vp C(1+\vp^{-1} \vp'(N) ) ( \sum |a_j|^2)^{1/2}$$
  where $\vp'(N)$ is independent of $\vp$ and satisfies $\vp'(N)\to 0$ when $N\to \infty$.
   Clearly \eqref{a4} and \eqref{a44}  follow.\\ Let $R_N=4C(1+\vp^{-1} \vp'(N) ) ( \sum |a_j|^2)^{1/2}$.
 By Tchebyshev's inequality $(E\|\hat S^{(N)} \|^p) ^{1/p} \le e^\vp R_N$ implies
  $$\P\{ \|\hat S^{(N)} \| > e^{2\vp} R_N\}\le     \exp -\vp p=N^2.$$
  From this it is immediate  that almost surely
 $$ \limsup_{N\to \infty} \|\hat S^{(N)} \| \le e^{2\vp} 4 C( \sum |a_j|^2)^{1/2}$$
 and hence \eqref{a5} follows.    \end{proof}
   
    \begin{rem}\label{sim1}  The same argument can be applied when $a_j\in M_k$ for any integer $k>1$. 
    Then we find
    $$\limsup_{N\to \infty}\E \| \sum\nolimits_1^n  a_j \otimes U_j \otimes \bar U_j (1-P) \| \le   4C
  \max\{ \|\sum a_j^*a_j\|^{1/2}, \|\sum a_ja_j^*\|^{1/2}   \} .$$
Moreover we have 
almost surely
 $$\limsup_{N\to \infty} \| \sum\nolimits_1^n  a_j \otimes U_j \otimes \bar U_j (1-P) \| \le   4C
  \max\{ \|\sum a_j^*a_j\|^{1/2}, \|\sum a_ja_j^*\|^{1/2}   \} .$$   
    \end{rem}
        \begin{rem} The  preceding also allows us  to majorize double sums
        of the form
        $$\sum\nolimits_{i\not=j} a_{ij}\otimes U_i \otimes \bar U_j .$$
        Indeed, we have ${\cl E} (Y_i \otimes \bar Y_j)=(U_i \otimes \bar U_j) (\E|Y|\otimes \E|Y|)$ for any $i\not =j$, and
        there is a constant $b>0$ (independent of $N$) such that $\E|Y|\ge b I$.
        Therefore,   for any $p\ge 1$,   any  $k$, any sequence $(a_{ij})$ in $M_k$, and any $1\le q\le \infty$, we have
        $$\E\|\sum\nolimits _{i\not=j} a_{ij}\otimes U_i \otimes \bar U_j \|_q^p \le b^{-2p}  \E\|\sum\nolimits _{i\not=j} a_{ij}\otimes Y_i \otimes \bar Y_j \|_q^p\le 2^p b^{-2p}   \E\|\sum\nolimits _{i\not=j} a_{ij}\otimes Y_i \otimes \bar Y'_j \|_q^p. $$ \end{rem}
   \begin{rem}\label{sim2}  We refer the reader to  \cite[Theorem 16.6]{Pg} for a self-contained proof of \eqref{a1}
  for   double sums of the form $\sum _{i,j} a_{ij}  Y_i \otimes \bar Y'_j $  for scalar coefficients $a_{ij}$.
        \end{rem}
 \bigskip
 
    \n\textbf{Acknowledgment.}  Thanks to  Kate Juschenko for useful conversations at an early stage of this investigation. I am grateful  
    to Mikael de la Salle for simplifying the proof of Lemma \ref{lem77},
    and to him and   Yanqi Qiu for numerous corrections.
 Thanks also to S. Szarek for a useful suggestion on the appendix.

  \end{document}